\documentclass[12pt,reqno,a4paper]{amsart}
\usepackage{amssymb,amsfonts,latexsym,mathrsfs,tikz}
\title[Non-orthogonal realizations of Coxeter groups]{Non-orthogonal geometric realizations of Coxeter groups}

\author{Fu, Xiang}

\dedicatory{\upshape
School of Mathematics and Statistics\\
University of Sydney, NSW 2006, Australia\\[.5em]
\texttt{xifu9119@mail.usyd.edu.au}\\
\texttt{xiangf@maths.usyd.edu.au}\\[1em]
Preliminary version,
\today
}

\newtheorem{theorem}{Theorem}[section]
\newtheorem{lemma}[theorem]{Lemma}
\newtheorem{proposition}[theorem]{Proposition}
\newtheorem{corollary}[theorem]{Corollary}

\newtheorem{notation}[theorem]{Notation}
\theoremstyle{definition}
\newtheorem{definition}[theorem]{Definition}
\newtheorem{remark}[theorem]{Remark}
\newtheorem{remarks}[theorem]{Remarks}
\numberwithin{equation}{section}

\newcommand{\Z}{\mathbb{Z}}
\newcommand{\N}{\mathbb{N}}
\newcommand{\R}{\mathbb{R}}

\DeclareMathOperator{\PLC}{PLC}
\DeclareMathOperator{\coeff}{coeff}

\DeclareMathOperator{\supp}{supp}
\DeclareMathOperator{\spa}{span}
\DeclareMathOperator{\GL}{GL}
\DeclareMathOperator{\dep}{dp}
\DeclareMathOperator{\ord}{ord}
\DeclareMathOperator{\Sym}{Sym}
\DeclareMathOperator{\hm}{Hom}
\DeclareMathOperator{\Neg}{Neg}


\begin{document}

\begin{abstract}
We define in an axiomatic fashion a \emph{Coxeter datum} for an arbitrary Coxeter
group $W$. This Coxeter datum will specify a pair of reflection representations
of $W$ in two vector spaces linked only by a bilinear paring without any integrality 
or non-degeneracy requirements. These representations
are not required to be embeddings of $W$ in 
the orthogonal group of any vector space, and they give rise to a pair of inter-related root
systems generalizing the classical root systems of Coxeter groups. We
obtain comparison results between these non-orthogonal
root systems and the classical root systems. Further, we study the equivalent 
of the Tits cone in these non-orthogonal representations, and we show that 
strong results on the geometry in such cones can be obtained.  
\end{abstract}

\maketitle

\section{Introduction}
%This paper gives a unified method of constructing a pair of inter-related root systems
%for all Coxeter groups. 

The notion of \emph{root systems} is an essential tool in the study of Coxeter groups.
For an arbitrary Coxeter system $(W, R)$ in the sense of \cite{BN68} and \cite{HM}, its
root system $\Phi$ is a set of vectors arising from the \emph{Tits representation} of $W$
as a group generated by reflections with respect to some hyperplanes in a real vector space
$V$. Here the representation space $V$ is equipped with a symmetric bilinear form and has 
dimension equal to the cardinality of $R$, and $\Phi$ consists of a set of representative
normal vectors for the reflecting hyperplanes. The elements of $\Phi$ are called \emph{roots}, 
and those roots corresponding to elements of $R$ are known as \emph{simple roots}.  
Following the convention of \cite{DK94} and \cite{DK09}, we also use the 
term \emph{root basis} for the set of simple roots. Under this classical construction, 
the abstract Coxeter group $W$ is embedded in the orthogonal group of the chosen bilinear
form on $V$, and hence the resulting reflection representation may be referred to as 
\emph{orthogonal}. 

In \cite[Ch.V, \S 4]{BN68} and \cite[\S 5.3]{HM} simple roots are required to form a basis for $V$, 
and for each pair of simple roots the value taken by the bilinear form is completely determined by the order
of the product of the corresponding reflections. In particular, if the order is infinite, then
the bilinear form takes the value $-1$.

The orthogonal representations studied in \cite{BN68} and \cite{HM} gave 
a beautiful theory of root systems for finite Weyl groups. However, a certain integrality
condition on the bilinear form required there may not be necessarily satisfied in an arbitrary 
Coxeter groups. Furthermore, since reflection subgroups of 
Coxeter groups are themselves Coxeter groups, it seems desirable that the set of roots corresponding
to a reflection subgroup should itself constitute a root system. However, this requires relaxing
the conditions imposed on the simple roots. Specifically, since a proper reflection subgroup may have 
more Coxeter generators than the over group (as illustrated in \cite[Example 5.1]{CH11}),
linear independence of simple roots is not inherited by reflection subsystems. Moreover, 
the property that the bilinear form takes the value $-1$ whenever the product of the 
corresponding reflections has infinite order is not inherited by reflection subsystems either
(as illustrated in \cite[Example 1.1]{CH11}).  
   
The notion of a \emph{non-orthogonal} reflection representation for an infinite Coxeter group
$W$ was put to use in \cite{MP89} to study the Weyl groups of Kac-Moody Lie algebras. A notable
feature of this approach is the construction of a pair of root systems generalizing the classical
notion in two inter-related vector spaces $V_1$ and $V_2$. The resulting reflection representations
are not embeddings of $W$ in any orthogonal group, thus explaining the term ``non-orthogonal''. 
In \cite{MP89}, the vector spaces $V_1$ and $V_2$ are linked by a non-degenerate bilinear form
satisfying an integrality condition, and the resulting root systems are required to form  
integral lattices in $V_1$ and $V_2$ respectively. This approach gives a well developed 
theory of root systems for the Weyl groups of Kac-Moody Lie algebras, but 
it may not apply to general infinite Coxeter groups, as commented in the remark immediately following 
the definition of root bases in \cite[1.1.1]{PR09}.

In this paper, we give a generalization of the non-orthogonal reflection representations studied
in \cite{MP89}. In particular, we remove the non-degeneracy and integrality condition on the 
bilinear form. In this situation, the root systems need not to form integral lattices. Furthermore, 
in contrast to \cite{MP89}, the removal of the the non-degeneracy 
condition on the bilinear pairing implies that the two representation spaces $V_1$
and $V_2$ may not necessarily be identified as algebraic duals of each other.

This paper also aims to give detailed comparisons between the root systems arising from 
non-orthogonal representations and their classical counterparts. 
For general Coxeter groups, in the literature, it appears that little attention were paid 
to comparing the root systems arising from non-orthogonal representations to
their classical counterparts. This paper presents a collection of basic combinatorial and geometric results
in non-orthogonal root systems for general Coxeter groups.  
Specific to the approach taken in this paper, a number of basic results and their proofs, 
and in some cases even the statements, involve new ideas. For instance, some
of the resulting reflection representations have root systems which inevitably 
involve potentially infinitely many different (positive) roots corresponding to 
the same reflection; these appear in neither the Kac-Moody nor orthogonal Coxeter
settings.   
 
A number of studies on non-orthogonal representations of Coxeter groups exist in 
the literature, for example, \cite{BC01}, \cite{BC06} and \cite{PR09}. Our approach
presented here also generalize those. In particular, the paired representation
spaces in those are all required to be algebraic duals of each other, and in 
\cite{BC01} and \cite{BC06} the root bases are required to be linearly independent.
As observed in the classical orthogonal case, the removal of the linear independence 
condition imposed on root bases is essential in developing 
a unified theory of root systems applicable to reflection subgroups, and consequently, 
in this paper, we do not impose the linear independence condition on root bases.
%Removing the linear independence requirement on the root bases and the
%integrality condition on the bilinear pairing is essential in developing 
%a unified theory of root systems for general (infinite) Coxeter groups, especially, 
%with respect to their reflection subgroups.
%
%In a unified theory of root systems, 
%it is desirable that each reflection subgroup should yield a root subsystem
%satisfying all the axioms of a root system. 
%For reflection subgroups within Coxeter groups, two issues  
%arise if we insist on the linear independence requirement on the root bases and  
%the integrality condition on the bilinear pairing: 
%\begin{itemize}
%\item [(i)] The equivalent of the simple roots in the root system for a reflection subgroup
%as a subset of the root system of the over-group might not be linearly 
%independent, as seen in \cite[Example 1]{MP89}. Indeed, the linear independence
%requirement on root bases is not desirable since
%a proper reflection subgroup of a Coxeter group may have strictly more
%Coxeter generators than the over-group (an example of this fact can be found in \cite[Example 5.1]{CH11}).   
%\item [(ii)] The integrality condition may not be satisfied in general reflection subgroups or non-crystallographic 
%Coxeter groups, as observed in the remark immediately following the definition of root bases in \cite[1.1.1]{PR09}.
%\end{itemize}
%Consequently, we drop the linear independence and integrality requirements.
 
This paper begins with a section devoted to proving that after the removal of  
a number of strong requirements of a set of root data in the sense of \cite{MP89} 
we can still construct a faithful reflection representation of an abstract Coxeter group 
(Theorem~\ref{th:isomorphism}). Section~\ref{sec:rts&cc} collects a number of basic facts
about non-orthogonal root systems. In particular, the issue with specifying a preferred 
way of expressing roots as linear combinations of simple roots 
(ambiguity may arise, since simple roots are no longer required to be 
linearly independent), and the issue of the possibility
of many (positive) roots corresponding to the same reflection are addressed. 
More interesting facts are then given in Section ~\ref{sec:cwsgr} and 
Section~\ref{sec:nonpos}. In particular, Section~\ref{sec:cwsgr} provides comparison
results (mostly in the form of inequalities) between constants associated 
to non-orthogonal representations and those for a corresponding orthogonal 
representation. These should be of utility in further study of 
non-orthogonal representations by reduction to known results in the
orthogonal setting. Finally, Section~\ref{sec:nonpos} studies 
the equivalent of the \emph{Tits cone} in non-orthogonal 
representations. In particular, we generalize a well known and useful result
of G.~A.~Maxwell on the dual of the Tits cone from special elements 
in the dual of the Tits cone to arbitrary elements (Theorem~\ref{th:Tits}).

\section{Paired reflection representations}
\label{sec:prr}

Let $S$ be an arbitrary set in which each unordered pair $\{s, t\}$ of
elements is assigned an $m_{st} \in \Z \cup \{\infty\}$, subject to the
conditions that $m_{ss} =1$ (for all $s$ in $S$), and $m_{st}\geq 2$ (for all
distinct $s,t$ in $S$). Suppose that $V_1$ and $V_2$ are vector spaces over the
real field $\R$, and suppose that there exists a bilinear map 
$\langle\ ,\ \rangle :V_1 \times V_2 \to \R$ and sets 
$\Pi_1 =\{\ \alpha_s \ | \ s\in S \ \}\subseteq V_1$ and $\Pi_2 =\{\ \beta_s
\ | \ s\in S \ \}\subseteq V_2 $ such that the following conditions
hold:
\begin{itemize}
\item[(C1)]$\langle \alpha_s, \beta_s\rangle =1$, for all $s\in S$; 
\item [(C2)]$\langle \alpha_s, \beta_t \rangle \leq 0$, for all distinct $s,t \in S$;
\item[(C3)]for all $s,\,t\in S$,
$$
\langle \alpha_s, \beta_t \rangle \langle \alpha_t, \beta_s \rangle = 
\begin{cases}
\cos^2 (\pi/m_{st}) & \text{if } m_{st} \neq \infty,\\
 \gamma^2, \text{ for some $\gamma \geq 1$} & \text{if } m_{st}=\infty; 
\end{cases}
$$
\item[(C4)] $\langle \alpha_s, \beta_t \rangle =0 $ if and
only if $\langle \alpha_t, \beta_s \rangle =0$, for all $s, t \in S$;  
\item[(C5)] $\sum_{s\in S}\lambda_s\alpha_s=0$ with $\lambda_s\ge0$ for all~$s$
implies $\lambda_s=0$ for all~$s$, and $\sum_{s\in S}\lambda_s\beta_s=0$ with
$\lambda_s\ge0$ for all $s$ implies $\lambda_s=0$ for all~$s$.
\end{itemize}

Note that (C3) and (C4) together imply that $\langle \alpha_s, \beta_t \rangle$
and $\langle \alpha_t, \beta_s \rangle$ are zero if and only if $m_{st} =2$. We
can also express (C5) more compactly as $0 \notin \PLC(\Pi_1)$ and $0\notin \PLC(\Pi_2)$,
where $\PLC(A)$ (the \textit{positive linear combinations of $A$}) of any
set $A$ is defined to be 
$$
\{\,\sum_{a\in A} \lambda_a a\mid\text{$\lambda_a \geq 0$ 
for all $a\in A$, and $\lambda_{a'}>0$ for some $a' \in A$\,} \}.
$$
\begin{definition}
\label{df:coxdat}
In the above situation, if conditions (C1) to (C5) are satisfied then
we call $\mathscr{C}=(S,V_1,V_2,\Pi_1,\Pi_2,\langle\ ,\ \rangle)$ a \textit{Coxeter datum}.
The $m_{st}$ ($s,t\in S$) are called the \textit{Coxeter parameters\/} of $\mathscr{C}$.
\end{definition}

Following \cite[1.1.1]{PR09}, we call $\mathscr{C}$ \emph{free} if both
$\Pi_1$ and $\Pi_2$ are linearly independent. Throughout this paper,
$\mathscr{C}=(S,V_1,V_2,\Pi_1,\Pi_2,\langle\ ,\ \rangle)$
will be a fixed Coxeter datum with Coxeter parameters~$m_{st}$, unless otherwise stated.  
We stress that, in general, $\mathscr{C}$ is not required to be free.

\begin{remark}
In a Coxeter datum $\mathscr{C}$, if we take $V_2=\hm(V_1, \R)$ and take
$\langle\,,\,\rangle$
to be the natural pairing on $V_1\times \hm(V_1, \R)$, then the conditions (RB1), (RB2)
and (RB3) required in the definition of the \emph{root data} of \cite[1.1.1]{PR09} are
automatically satisfied in $\mathscr{C}$ (though, here we have scaled each element of
$\Pi_1$ and $\Pi_2$ by a factor of $\frac{1}{\sqrt{2}}$). 
Thus the root data in the sense of \cite{PR09} are special cases of  Coxeter
datum defined here. Further, if we assume that $\mathscr{C}$ is free then we recover 
a set of root data as given in \cite{hee91}, whereas if we assume that $\langle\,,\,\rangle$
is non-degenerate and
$\langle \alpha_s, \beta_t\rangle \in \Z$ for all $s, t\in S$ 
then we recover a set of root data as given in \cite{MP89}.
\end{remark}

\begin{definition}
\label{df:blf}
Given a Coxeter datum $\mathscr{C}=(S,V_1,V_2,\Pi_1,\Pi_2,\langle\ ,\ \rangle)$, 
for each $s\in S$, let
$\rho_{V_1}(s)$ and $\rho_{V_2}(s)$ 
be the linear transformations on $V_1$ and $V_2$  defined by
\begin{align*}
               \rho_{V_1}(s)(x)&= x-2\langle x, \beta_s \rangle \alpha_s\\ 
\noalign{\vskip-5pt\hbox{for all $x\in V_1$, and}\vskip-5pt}
               \rho_{V_2}(s)(y)&= y-2\langle \alpha_s , y \rangle \beta_s
\end{align*}
for all $y\in V_2$.
For each $i\in\{1,2\}$  let $R_i(\mathscr{C}) =\{\,\rho_{V_i}(s) \ | \ s\in S\,\}$, 
and let $W_i(\mathscr{C})$ be the subgroup of $\GL(V_i)$ generated by $R_i$.
\end{definition}

For each $s\in S$, it is readily checked that $\rho_{V_1}(s)$ and $\rho_{V_2}(s)$ 
are involutions with $\rho_{V_1}(s)(\alpha_s)=-\alpha_s$
and $\rho_{V_2}(s)(\beta_s)=-\beta_s$. Further, we have: 
\begin{proposition}
\label{pp:p0}
Let $x\in V_1$ and $y\in V_2$. Then for all
$s\in S$,
$$
\langle \rho_{V_1}(s)(x),\rho_{V_2}(s)(y)\rangle=\langle x,y\rangle. 
$$
\qed
\end{proposition}

Though $\mathscr{C}$ lacks freeness in general, nevertheless, 
conditions (C1), (C2) and (C5) of the definition of a Coxeter datum together yield that:

\begin{lemma}
\label{lem:l1}
For each $s\in S$ we have $\alpha_s \notin
\PLC(\Pi_1\setminus\{\alpha_s\})$, and $\beta_s \notin \PLC(\Pi_2\setminus
\{\beta_s\})$. In particular, for distinct $s, t\in S$, the set $\{\alpha_s,
\alpha_t\}$ is linearly independent, and so is $\{\beta_s, \beta_t \}$.
\qed
\end{lemma}

For each $i\in \{1, 2\}$, the above lemma yields that
$\rho_{V_i}(s)\neq\rho_{V_i}(t)$ whenever $s,t\in S$ are distinct. 
The following lemma, readily obtained from direct calculations, summarizes a few useful results:
\begin{lemma}
\label{lem:storder}
(i)\quad Suppose that $s,t \in S$ such that $m_{st} \notin \{1, \infty\}$, and 
let $\theta =\pi/m_{st}$. If $m_{st}\neq 2$ then for each $n\in \N$, 
\begin{align*}
(\rho_{V_1}(s)\rho_{V_1}(t))^n (\alpha_s)
&= \frac{\sin(2n+1)\theta}{\sin\theta}\alpha_s +\frac{-\cos\theta}{\langle
\alpha_t, \beta_s\rangle}\,\frac{\sin(2n\theta)}{\sin\theta} \alpha_t  \\
\noalign{\hbox{and}}
(\rho_{V_2}(s)\rho_{V_2}(t))^n (\beta_s)
&= \frac{\sin(2n+1)\theta}{\sin\theta}\beta_s +\frac{-\cos\theta}{\langle
\alpha_s, \beta_t\rangle}\,\frac{\sin(2n\theta)}{\sin\theta} \beta_t.  
\end{align*}
While if $m_{st}=2$ then for each $n\in \N$, 
$$(\rho_{V_1}(s)\rho_{V_1}(t))^n(\alpha_s) =(-1)^{n}\alpha_s, \quad\text{and }\quad 
(\rho_{V_2}(s)\rho_{V_2}(t))^n(\beta_s) =(-1)^{n}\beta_s.$$

\noindent(ii)\quad Suppose that $s, t \in S$ such that $m_{st} =\infty$.
Define $\theta  =\cosh^{-1}(\gamma)$, where
$\gamma=\sqrt{\langle \alpha_s,\beta_t\rangle\langle \alpha_t,\beta_s\rangle}$.
Then for each $n\in \N$, 
%\begin{align*}
\begin{equation*}
 (\rho_{V_1}(s)\rho_{V_1}(t))^n (\alpha_s)
=
\begin{cases}
 \frac{\sinh(2n+1)\theta}{\sinh \theta}\alpha_s +\frac{-\gamma}{\langle
\alpha_t, \beta_s\rangle}\frac{\sinh(2n\theta)}{\sinh \theta} \alpha_t, \text{ if $\theta\neq 0$}, \\
(2n+1)\alpha_s -\frac{2n\gamma}{\langle \alpha_t, \beta_s\rangle}\alpha_t, \text{ }
\text{ }\text{ }\text{ }\text{ }\text{ }\text{ }\text{ }\text{ }\text{ }\text{ }\,\text{if $\theta=0$};\\
\end{cases}
%\noalign{\hbox{and}}
\end{equation*}
and
\begin{equation*}
(\rho_{V_2}(s)\rho_{V_2}(t))^n (\beta_s)
=
\begin{cases} 
\frac{\sinh(2n+1)\theta}{\sinh \theta}\beta_s +\frac{-\gamma}{\langle
\alpha_s, \beta_t\rangle}\frac{\sinh(2n\theta)}{\sinh \theta} \beta_t, \text{ if $\theta\neq 0$}, \\
(2n+1)\beta_s -\frac{2n\gamma}{\langle \alpha_s, \beta_t\rangle}\beta_t, \text{ }\text{ }\text{ }
\text{ }\text{ }\text{ }\text{ }\text{ }\text{ }\text{ }\text{ }\,\text{if $\theta=0$}.
\end{cases}
\end{equation*}
%\end{align*}
\qed
\end{lemma}

\begin{remarks}
\label{rmk:st} Let $\ord(\rho_{V_i}(s)\rho_{V_i}(t))$ 
denote the order of $\rho_{V_i}(s)\rho_{V_i}(t)$ in $\GL(V_i)$, for each $i\in \{1, 2\}$. Then: 

\noindent\rm{(i)}\quad Lemma \ref{lem:storder} (i) yields that
$\ord(\rho_{V_i}(s)\rho_{V_i}(t))\geq m_{st}$ when $m_{st}\neq \infty$. Indeed, in the subspace 
with basis $\{ \alpha_s,\alpha_t \}$ the following $m_{st}$ elements 
$$
\alpha_s,(\rho_{V_1}(s)\rho_{V_1}(t))\alpha_s,(\rho_{V_1}(s)\rho_{V_1}
(t))^2\alpha_s,\ldots,(\rho_{V_1}(s)\rho_{V_1}(t))^ {m_{st}-1}\alpha_s$$ are all
distinct, and the same  holds in the subspace with basis $\{\beta_s, \beta_t\}$.

\noindent\rm{(ii)}\quad In a similar way as the above, it follows from  Lemma
\ref{lem:storder}~(ii) that 
$\ord(\rho_{V_i}(s)\rho_{V_i}(t))= \infty$ when $m_{st}=\infty$. 

\end{remarks}

The above observations naturally lead to the following:
\begin{proposition}
\label{pp:order}
Suppose that $s, t \in S$. Then for each $i\in \{1, 2\}$,
$$\rho_{V_i}(s)\rho_{V_i}(t) \text{ has order $m_{st}$ in
$\GL(V_i)$}.$$
\end{proposition}
\begin{proof} We give a proof that $\rho_{V_1}(s)\rho_{V_1}(t)$ has order
$m_{st}$ in $\GL(V_1)$ below, and we stress that the same argument 
will hold for $\rho_{V_2}(s)\rho_{V_2}(t)$ in $\GL(V_2)$. 
Observe that we only need to consider the cases 
when $m_{st}\notin\{1, \infty\}$, for the statement of 
the proposition follows readily from Remark~\ref{rmk:st}~(ii) and the fact that  
each $\rho_{V_1}(s)$, for all $s\in S$, is an involution.
Next let $\alpha \in V_1$ be
arbitrary. Then
\begin{align}
\label{eq:ord}
(\rho_{V_1}(s)\rho_{V_1}(t))(\alpha)
&= \rho_{V_1}(s)(\alpha -2\langle \alpha, \beta_t \rangle \alpha_t)\notag\\
&=\alpha-2\langle \alpha, \beta_s \rangle \alpha_s -2\langle \alpha, \beta_t
\rangle (\alpha_t-2\langle \alpha_t, \beta_s\rangle \alpha_s)\notag\\
&=\alpha + (4\langle \alpha, \beta_t \rangle \langle \alpha_t, \beta_s \rangle
-2\langle \alpha, \beta_s\rangle)\alpha_s -2\langle \alpha, \beta_t\rangle
\alpha_t. 
\end{align}
If $m_{st}=2$, then (\ref{eq:ord}) yields that
$$(\rho_{V_1}(s)\rho_{V_1}(t))(\alpha)=(\rho_{V_1}(t)\rho_{V_1}(s))(\alpha)
=\alpha-2\langle \alpha, \beta_s\rangle \alpha_s-2\langle \alpha, \beta_t\rangle \alpha_t,$$
so that $\rho_{V_1}(s)$ and $\rho_{V_1}(t)$ commute. Hence $\ord(\rho_{V_1}(s)\rho_{V_1}(t))=2$
when $m_{st}=2$, and it remains to check the case when $m_{st}>2$.
Observe that if $\alpha=\alpha_s$ and $\alpha=\alpha_t$, then (\ref{eq:ord}) yields that
\begin{align*}
(\rho_{V_1}(s)\rho_{V_1}(t))(\alpha_s) &=
(4\cos^2(\pi/m_{st})-1)\alpha_s -2\langle \alpha_s, \beta_t\rangle
\alpha_t \\
\noalign{\hbox{and}}
(\rho_{V_1}(s)\rho_{V_1}(t))(\alpha_t) &= 2\langle \alpha_t, \beta_s\rangle
\alpha_s -\alpha_t.
\end{align*}
Therefore the action of $\rho_{V_1}(s)\rho_{V_1}(t)$ on $\R\{\alpha_s,
\alpha_t,\alpha\}$ may be represented by the following matrix $M$:
\begin{equation*}
M=\left(
\begin{array}{ccc}
4\cos^2(\pi/m_{st})-1 & 2\langle \alpha_t, \beta_s \rangle & 4\langle \alpha,
\beta_t\rangle \langle \alpha_t, \beta_s \rangle -2\langle \alpha, \beta_s
\rangle \\
-2\langle \alpha_s, \beta_t \rangle & -1 & -2\langle \alpha, \beta_t \rangle \\
0 & 0 & 1
\end{array} \right).
\end{equation*}
It is readily checked that $M$ has distinct eigenvalues
$e^{i\tfrac{2\pi}{m_{st}}}$, $e^{-i\tfrac{2\pi}{m_{st}}}$ and~$1$. Hence $M$ has
order $m_{st}$, and so $(\rho_{V_1}(s)\rho_{V_1}(t))^{m_{st}} =1$ in
$\GL(V_1)$. Finally, in
view of Remark \ref{rmk:st} (i), it follows that
$\ord(\rho_{V_1}(s)\rho_{V_1}(t))$ is precisely $m_{st}$.
\end{proof}

\begin{remark}
\label{rmk:isomorphism}
Given a Coxeter datum $\mathscr{C}=(\,S, V_1, V_2, \Pi_1, \Pi_2, \langle, \rangle\,)$ 
with Coxeter parameters $m_{st}$ (where $s,t\in S$),
let $(W, R)$ be a Coxeter system in the sense of \cite{HH81} or \cite{HM} with
$R=\{\, r_s\mid s\in S\,\}$ being a set of involutions generating $W$ 
subject only to the condition that  the order of the product $r_s r_t$ is $m_{st}$ whenever $s, t$
are in $S$ with $m_{st}\neq \infty$. Then Proposition~\ref{pp:order} yields
that there are group homomorphisms $f_1\colon W\to W_1(\mathscr{C})$ and
$f_2\colon W\to W_2(\mathscr{C})$ satisfying $f_1(r_s)=\rho_{V_1}(s)$ and
$f_2(r_s)=\rho_{V_2}(s)$ for all $s\in S$.
\end{remark}
The principal result of this section
is the following:
\begin{theorem}
\label{th:isomorphism}
Let $(W, R)$, $f_1$ and $f_2$ be as the above. Then $f_1$ and $f_2$ are
isomorphisms; that is, $(W_1(\mathscr{C}), R_1(\mathscr{C}))$ and 
$(W_2(\mathscr{C}), R_2(\mathscr{C}))$ are both Coxeter systems
isomorphic to $(W, R)$.
\end{theorem}

For each $i\in \{1, 2\}$, since $W_i(\mathscr{C})$  is generated by the elements
of $R_i(\mathscr{C})$, it follows readily that each $f_i$ is surjective. 
Thus  only the injectivity of 
$f_i$  needs to be checked. Before we can do so, a few elementary 
results are needed. First we have a result easily obtained from the formulas 
in Lemma~\ref{lem:storder}:

\begin{lemma}
\label{lem:coeffpos}
Suppose that $s,t \in S$, and let $n$ be an integer such that $0\le n<m_{st}$.
Let $\lambda_n$, $\mu_n$, $\lambda'_n$ and $\mu'_n$ be constants such that
\begin{align*}
&\underbrace{\cdots \rho_{V_1}(t)\rho_{V_1}(s)\rho_{V_1}(t)}
_\text{$n$ factors } (\alpha_s) = \lambda_n \alpha_s + \mu_n \alpha_t \\
\noalign{\hbox{and}}
&\underbrace{\cdots \rho_{V_1}(s)\rho_{V_1}(t)\rho_{V_1}(s)}
_\text{$n$ factors } (\alpha_t) = \lambda'_n \alpha_s + \mu'_n \alpha_t. 
\end{align*}
Then all four constants $\lambda_n$, $\mu_n$, $\lambda'_n$ and $\mu'_n$ are 
non-negative.
\qed
\end{lemma}

\begin{remark}
The same argument applies equally well if we replace in the above lemma,
respectively,
$\rho_{V_1}(s)$, $\rho_{V_1}(t)$, $\alpha_s$ and $\alpha_t$ by $\rho_{V_2}(s)$,
$\rho_{V_2}(t)$, $\beta_s$ and $\beta_t$.
\end{remark}

Let $W$, $f_1$ and $f_2$ be as in Remark \ref{rmk:isomorphism}.
Then $f_1$ and $f_2$ give rise to $W$-actions on $V_1$ and $V_2$ in the
following way: $wx =(f_1 (w))(x)$ for all $w\in W$ and $x\in V_1$, and $wy
=(f_2(w))(y)$ for all $w\in W$ and $y\in V_2$.
Let $\ell\colon W\to \N$ be the \emph{length} function of $W$ with respect to
$R$. For $w\in W$  we say that an expression of the form 
$w=r_{s_1}\cdots r_{s_l}$ (where $s_1, \ldots s_l\in S$) is \emph{reduced} if
$\ell(w)=l$. For any $w\in W$, an easy induction on $\ell(w)$ 
yields the following extension to Proposition \ref{pp:p0}:
\begin{lemma}
\label{lem:eqf}
Given a Coxeter datum $\mathscr{C}=(\,S, V_1, V_2, \Pi_1, \Pi_2, \langle,\rangle\,)$,
and let $W$ be as in in Remark \ref{rmk:isomorphism}. Then
$\langle x, y \rangle = \langle wx, wy
\rangle$ for
all $w\in W$, $x\in V_1$, and $y\in V_2$.
\qed
\end{lemma}

\begin{proposition}
%\textup{(\cite[Theorem, Lecture 1]{RB96})}
\label{pp:l&p}
Let $W$ be as the above, and let $w\in W$ and $s\in S$. 
If $\ell(wr_s)\geq \ell(w)$ then 
$w \alpha_s \in \PLC(\Pi_1)$.
\end{proposition}
\begin{proof}
Choose $w\in W$ of minimal length such that the assertion fails for some 
$\alpha_s \in \Pi_1$, and choose such an $\alpha_s$. Certainly $w \neq 1$, 
since $1\alpha_s =\alpha_s$ is trivially a positive linear combination of 
$\Pi_1$. Thus $\ell(w) >1$, and we may choose $t\in S$ such that $w_1 = w r_t$ 
has length $\ell(w)-1$. If $\ell(w_1r_s)\geq \ell(w_1)$, then $\ell(w_1 r) \geq \ell(w_1)$ 
for both $r=r_s$ and $r=r_t$. Alternatively, if $\ell(w_1 r_s) < \ell(w_1)$, we 
define $w_2 = w_1 r_s$, and note that $\ell(w_2 r)\geq \ell(w_2)$ will hold for 
 $r=r_s$ and $r=r_t$ if $\ell(w_2 r_t) \geq \ell(w_2)$. If this latter condition 
 is not satisfied then we define $w_3 = w_2 r_t$. Continuing in this way we 
 find, for some positive integer $k$, a sequence of elements 
 $w_0 = w, w_1, w_2, \cdots, w_k$ with $\ell(w_i)=\ell(w)-i$ for all 
 $i = 0, 1, 2, \cdots, k$, and, when $i<k$,
\begin{equation*}
w_{i+1} = \left\{
\begin{array}{rl}
w_i r_s & \text{if $i$ is odd } \\
w_i r_t & \text{if $i$ is even } 
\end{array} \right.
\end{equation*}
Now since $0\leq \ell(w_k)=\ell(w)-k$, we conclude that $\ell(w)$ is an upper 
bound for the possible values of $k$. Choosing $k$ to be as large as 
possible, we deduce that $\ell(w_k r) \geq \ell(w_k)$ for both $r=r_s$ and $r= r_t$, 
for otherwise the process described above would allow a $w_{k+1}$ to be found, 
contrary to the definition of $k$. By the minimality of our original 
counterexample it follows that $w \alpha_s$ and $w\alpha_t$ are both in $\PLC(\Pi_1)$.

We have $w = w_k v$, where $v$ is an alternating product of $r_s$'s and 
$r_t$'s, ending in $r_t$, and with $k$ factors altogether. Obviously 
this means that $\ell(v)\leq k$. But $w =w_k v$ gives $\ell(w) \leq \ell(w_k)+\ell(v)$, 
so it follows that $\ell(v)\geq \ell(w)-\ell(w_k) =k$, and hence $\ell(v)=k$. Furthermore, 
in view of the hypothesis that $\ell(wr_s) \geq \ell(w)$, and since $w_k v r_s =wr_s$, we have 
$$\ell(w_k) +\ell(vr_s)\geq \ell(wr_s)\geq \ell(w) = \ell(w_k)+k =\ell(w_k)+\ell(v), $$
and hence $\ell(vr_s)\geq \ell(v)$. In particular, $v$ cannot have a reduced 
expression in which the final factor is $r_s$, for if so $vr_s$ would 
have a strictly shorter expression.

Since $r_s$ and $r_t$ satisfy the defining relations of the dihedral 
group of order $2m_{st}$, it follows that every element of the subgroup 
generated by $r_s$ and $r_t$ has an expression of length less than $m_{st}+1$ 
as an alternating product of $r_s$ and $r_t$. Thus $\ell(v)\leq m_{st}$. 
Moreover, if $m_{st}$ is finite then the two alternating products of length 
$m_{st}$ define the same element; so $\ell(v)$ cannot equal $m_{st}$, as $v$ 
has no reduced expression ending with~$r_s$.  Thus Lemma \ref{lem:coeffpos} above yields
$v \alpha_s = \lambda_1 \alpha_s + \mu_1 \alpha_t$ for some non-negative 
coefficients $\lambda_1$ and $\mu_1$. Hence 
$$w\alpha_s = w_k v \alpha_s = w_k (\lambda_1 \alpha_s + \mu_1 \alpha_t)= \lambda_1 w_k \alpha_s + \mu_1 w_k \alpha_t \in \PLC(\Pi_1),$$
since $w_k \alpha_s, w_k \alpha_t \in \PLC(\Pi_1)$. 
This contradicts our original choice of $w$ and $\alpha_s$ as a 
counterexample to the statement of the proposition; so if $w\in W$, $s \in S$, 
with $\ell(wr_s)\geq \ell(w)$, then $w\alpha_s \in \PLC(\Pi_1)$.
\end{proof}

Now we are ready to complete the proof of Theorem \ref{th:isomorphism}:

\begin{proof}[Proof of Theorem \ref{th:isomorphism}]
Suppose, for a contradiction, that the kernel of $f_1$ is nontrivial, and choose
$w$ in the kernel of $f_1$ with $w \neq 1$. Then $\ell(w)>0$, and we may write
$w =w' r_s$ for some $s\in S$ and $w'\in W$ with $\ell(w')=\ell(w)-1$. Since 
$\ell(w'r_s)> \ell(w')$, Proposition~\ref{pp:l&p} yields 
$w'\alpha_s \in \PLC(\Pi_1)$. But then
$$
\alpha_s = w \alpha_s = (w' r_s)\alpha_s =w'(r_s \alpha_s)
=w'(-\alpha_s)=-w'\alpha_s,
$$
and hence $0=\alpha_s + w'\alpha_s \in \PLC(\Pi_1)$, contradicting condition
(C5) of a Coxeter datum. 
In an entirely similar way it can be shown that the kernel of $f_2$ is trivial.
\end{proof}

Let $W$ and $R$ be as in Remark \ref{rmk:isomorphism}. We call $W$
the \emph{abstract Coxeter group} determined by the
Coxeter parameters of the Coxeter datum~$\mathscr{C}$, and we call $(W, R)$ the 
\emph{abstract Coxeter system} associated with $\mathscr{C}$. 
Observe that Theorem~\ref{th:isomorphism} yields that the $W$-actions on $V_1$ and $V_2$
induced by the isomorphisms $f_1$ and $f_2$ are faithful.

\section{Root Systems and Canonical Coefficients}
\label{sec:rts&cc}
%Let $\mathscr{C}=(\,S, V_1, V_2,\Pi_1, \Pi_2, \langle\,,\,\rangle\,)$ be a Coxeter
%datum, and let $(W, R)$ be the associated abstract Coxeter system, 
%and retain all the notations of the previous section. The Coxeter datum
%$\mathscr{C}$ will give rise to a pair of suitably defined
%root systems, one in $V_1$ and the other in $V_2$. 

\begin{definition}
\label{df:rootsys}
Suppose that $\mathscr{C}=(\,S, V_1, V_2,\Pi_1, \Pi_2, \langle\,,\,\rangle\,)$ is a 
Coxeter datum, and suppose that $(W, R)$ is the associated abstract Coxeter system.

\noindent\rm{(i)} Define $\Phi_1(\mathscr{C})=W\Pi_1=\{\, w\alpha_s\mid  w\in W, s\in S   \,\}$, 
and similarly define $\Phi_2(\mathscr{C}) = W \Pi_2=\{\,w \beta_s \mid w\in W, s\in S \,\}$.  For each
$i\in\{1,2\}$, we call $\Phi_i(\mathscr{C})$ the \emph{root system} of $W$ in $V_i$, and its
elements the \emph{roots} of $W$ in $V_i$. We call $\Pi_i$ the set of
\emph{simple roots} in $\Phi_i(\mathscr{C})$, and we say that $\Pi_i$ forms a \emph{root
basis} for $\Phi_i(\mathscr{C})$.  

\noindent\rm{(ii)} Set $\Phi_i^+(\mathscr{C})= \Phi_i(\mathscr{C}) \cap
\PLC(\Pi_i)$, and $\Phi_i^-(\mathscr{C})=-\Phi_i^+(\mathscr{C})$ for each $i\in\{1,2\}$. 
We call $\Phi_i^+(\mathscr{C})$ the set of \emph{positive roots} in $\Phi_i(\mathscr{C})$, 
and $\Phi_i^-(\mathscr{C})$ the set of \emph{negative roots} in $\Phi_i(\mathscr{C})$. 
\end{definition}

Given the above notation, for each $i\in\{1,2\}$, we adopt the traditional diagrammatic description of
simple roots $\Pi_i$: draw a graph that has one vertex for each $s\in S$, and
join the vertices corresponding to $s,t \in S$ by an edge labelled by $m_{st}$
if $m_{st} >2$. The label $m_{st}$ is often omitted if $m_{st}=3$. Thus the
diagram 
\lower13 pt\hbox{\begin{tikzpicture}[x=.3cm, y=0.3cm, node distance=0mm]
%\node [coordinate](r) at (1,1){$\bullet$};
\node [rectangle](rlab) at (0,1){$\scriptstyle r$};
%\node [coordinate](s) at (0,-1){$\bullet$};
\node [rectangle](slab) at (-1,-1){$\scriptstyle s$};
%\node [coordinate](t) at (2,-1){$\bullet$};
\node [rectangle](tlab) at (3,-1){$\scriptstyle t$};
\draw (1,1) node{$\bullet$} -- (0,-1) node{$\bullet$} -- (2,-1) node{$\bullet$} 
-- (1,1);
\end{tikzpicture}}
\begingroup\advance\abovedisplayskip-2 pt\advance\belowdisplayskip-2 pt
corresponds to $\Pi_1 =\{\,\alpha_r, \alpha_s, \alpha_t \,\}$ and $\Pi_2 =\{\,
\beta_r, \beta_s, \beta_t\,\}$. 
Suppose that 
\begin{align*}
&\langle \alpha_r, \beta_s \rangle =-1/4,\quad &\langle\alpha_s, \beta_t\rangle =-1/6,\qquad\quad  
&\langle \alpha_t, \beta_r\rangle=-1/10,\\
&\langle \alpha_s, \beta_r \rangle =-1,\quad &\langle \alpha_t, \beta_s\rangle =-3/2,\qquad\quad 
&\langle \alpha_r, \beta_t\rangle=-5/2.
\end{align*}
Then
\begin{displaymath}
\begin{split}
&r_s \alpha_r=\alpha_r-2\langle \alpha_r, \beta_s \rangle \alpha_s =
\alpha_r+\tfrac{1}{2} \alpha_s;\\
&(r_r r_s) \alpha_r= r_r(\alpha_r+\tfrac{1}{2} \alpha_s)=\tfrac{1}{2}
\alpha_s;\\
&(r_t r_r r_s) \alpha_r =r_t (\tfrac{1}{2} \alpha_s)=\tfrac{1}{2} \alpha_s+
\tfrac{1}{6} \alpha_t;\\
&(r_s r_t r_r r_s)\alpha_r=r_s(\tfrac{1}{2} \alpha_s+ \tfrac{1}{6}
\alpha_t)=\tfrac{1}{6} \alpha_t;\\
&(r_r r_s r_t r_r r_s) \alpha_r=r_r (\tfrac{1}{6} \alpha_t)= \tfrac{1}{6}
\alpha_t+ \tfrac{1}{30} \alpha_r;\\
&(r_t r_r r_s r_t r_r r_s)\alpha_r=r_t(\tfrac{1}{6} \alpha_t+ \tfrac{1}{30}
\alpha_r)=\tfrac{1}{30}\alpha_r.
\end{split}
\end{displaymath}
%\endgroup
In particular, we notice from the above that it is possible for a non-trivial positive
scalar multiple of a root to also be a root, lying in the same $W$-orbit as
the root itself. Clearly if $w\alpha=\lambda\alpha$ where $\alpha \in \Phi_1(\mathscr{C})$ 
then $w^n\alpha=\lambda^n\alpha$ for all $n\in \N$.  Since it is quite possible that $\lambda\ne \pm 1$, 
it follows that there could well be infinitely
many non-trivial scalar multiples of $\alpha$ in~$\Phi_1(\mathscr{C})$. Further, all roots in the
$W$-orbit of $\alpha$ will possess this same property. Of course, the same situation
could arise in $\Phi_2(\mathscr{C})$ as well. This is one of the features setting $\Phi_1(\mathscr{C})$ 
and $\Phi_2(\mathscr{C})$ apart from the classical root systems studied in \cite[Ch.V]{BN68}, \cite{HH81} or \cite[Ch.5]{HM}.  
%Despite the fact that in 
%this regard $\Phi_1(\mathscr{C})$ and $\Phi_2(\mathscr{C})$ are different from root systems defined in 
%\cite{BN68}, \cite{HH81} or \cite{HM}, it nevertheless turns out that all major 
%properties of root systems can be generalized to our present setting.
%Indeed, similar to \cite{HM}, Proposition~\ref{pp:l&p} yields the following:

Proposition~\ref{pp:l&p} yields the following:
\begin{lemma}
\label{lem:key}
Suppose that $\mathscr{C}=(\,S, V_1, V_2, \Pi_1, \Pi_2, \langle\,,\,\rangle\,)$ is a Coxeter datum, 
and suppose that $(W, R)$ is the associated abstract Coxeter system.
\noindent\rlap{\rm (i)}\qquad
$\Phi_i(\mathscr{C}) = \Phi_i^+ (\mathscr{C})\uplus \Phi_i^-(\mathscr{C})$, for each $i\in \{1,2\}$, where $\uplus$
denotes disjoint union.

\noindent\rlap{\rm (ii)}\qquad
If $w\in W$ and $s\in S$, then 
\begin{equation*}
\ell(wr_s) =
\begin{cases}
\ell(w)+1  \text{   if } w\alpha_s \in \Phi_1^+(\mathscr{C}) \text{, and } w\beta_s \in
\Phi_2^+(\mathscr{C}),\\
\ell(w)-1  \text{   if } w\alpha_s \in \Phi_1^-(\mathscr{C}) \text{, and } w\beta_s \in
\Phi_2^-(\mathscr{C}).
\end{cases}
\end{equation*}\qed
\end{lemma} 

\begin{remark}
\label{rm:sign}
Part (ii) of the above lemma implies that for $w\in W$ and $s\in S$, 
$w\alpha_s \in \Phi_1^+(\mathscr{C})$ if and only if $w\beta_s \in \Phi_2^+(\mathscr{C})$, 
and $w\alpha_s \in \Phi_1^-(\mathscr{C})$ if and only if $w\beta_s \in \Phi_2^-(\mathscr{C})$.
\end{remark}

\begin{remark}
Since $\mathscr{C}$ is not assumed to be free, although we know
from Lemma~\ref{lem:key}~(i) that  each root in $\Phi_i(\mathscr{C})$ (for each $i\in \{1, 2\}$) 
is expressible as a linear combination of simple roots from $\Pi_i$  with coefficients all being 
of the same sign, these expressions need not be unique. Thus the concept of 
\emph{the coefficient of a simple root in a given root} is potentially ambiguous.
To obtain a canonical way of expressing roots in terms of simple roots, 
we employ a  construction similar as those of \cite{FU1}, \cite{HT97} and \cite{MP89}. We define 
a free Coxeter datum $\mathscr{C'}$ on the same set of Coxeter parameters as those 
of $\mathscr{C}$. Then both $\mathscr{C}$ and the free Coxeter datum 
$\mathscr{C'}$ correspond to the same abstract Coxeter system $(W, R)$. It 
turns out that for each $i\in \{1, 2\}$, there exists a canonical $W$-equivariant 
bijection $\pi_i\colon \Phi_i (\mathscr{C}) \leftrightarrow\Phi_i(\mathscr{C'})$ which 
maps simple roots to simple roots. Since each 
$\Phi_i(\mathscr{C'})$ is free, there is no ambiguity of the coefficient of a simple 
root in any given root of $\Phi(\mathscr{C'})$, and this will provide a canonical expression of roots in 
$\Phi_i(\mathscr{C})$ in terms of $\Pi_i$ via $\pi_i$.
However, since $\mathscr{C}$ lacks the integrality condition assumed in \cite{MP89}, 
the proof used in \cite{MP89} will not apply here. 
Further, since it is no longer true that there exists a 
bijection between $\Phi_i(\mathscr{C})^+$ ($i=1, 2$) and the reflections in $W$, 
the proofs used in \cite{FU1} and \cite{HT97} will not apply either. 
\end{remark}

Let $V_1'$ be a vector space over $\R$ with basis $\Pi_1'= \{ \,\alpha_s' \mid 
s\in S \,\}$ in bijective correspondence with~$S$, and let $V'_2$ be a vector 
space over $\R$ with basis $\Pi'_2= \{\, \beta'_s \mid s\in S\,\}$, also in 
bijective correspondence with~$S$. Define linear maps $\pi_1\colon V'_1 \to V_1$ 
and $\pi_2\colon V'_2 \to V_2$ by requiring that 
$$
\pi_1\bigl(\sum\limits_{s\in S}\lambda_s \alpha'_s\bigr)
=\sum\limits_{s\in S}\lambda_s \alpha_s, \quad \text{and} \quad
\pi_2\bigl(\sum\limits_{s\in S}\mu_s \beta'_s\bigr)=\sum\limits_{s\in S}\mu_s \beta_s,
$$
for  all $\lambda_s, \mu_s \in \R$, and
define a bilinear map $\langle \, ,\,\rangle'\colon V'_1\times V'_2 \to \R$ 
by requiring that 
$\langle \alpha'_s, \beta'_t \rangle' =\langle \alpha_s, \beta_t\rangle$
for all $s,\,t\in S$. Observe that then
$\langle x',y'\rangle'=\langle\pi_1(x'),\pi_2(y')\rangle$
for all $x'\in V_1'$ and $y'\in V_2'$.

Thus $\mathscr{C}'=(\,S,V_1',V_2',\Pi_1',\Pi_2',\langle\ ,\ \rangle'\,)$ 
is a free Coxeter datum with the same parameters as $\mathscr{C}$, 
and therefore is associated to the same abstract
Coxeter system $(W, R)$. Applying Theorems \ref{th:isomorphism} to
 $\mathscr{C}'$ then yields isomorphisms
$f'_1\colon W\to W'_1(\mathscr{C'})$ and $f'_2\colon W\to W'_2(\mathscr{C'})$.
These isomorphisms induce $W$-actions on $V_1'$ and $V_2'$ via:
$wx'=(f'_1(w))(x')$, and $u y'=(f'_2(u))(y')$ for all $w, u\in W$, $x'\in
V_1'$ and $y'\in V_2'$. Note that for each $i\in \{1, 2\}$ and each $s\in S$,  
it follows from the above definitions that $f'_i(r_s)=\rho_{V'_i}(s)$, and this in turn yields that
$\pi_i f_i'(r_s)=f_i(r_s)\pi_i,$
where $f_1$ and $f_2$ are as in Theorem~\ref{th:isomorphism}. Since $W$ is
generated by $\{\,r_s\mid s\in S\,\}$, it follows that 
$\pi_i f_i'(w)=f_i(w)\pi_i$ for all $w\in W$ and $i\in\{1,2\}$. Summing up, we have:
\begin{equation}\label{eq:homs}
\pi_i(wz')=w\pi_i(z')\qquad\text{for all $w\in W$ and $z'\in V_i'$},
\end{equation}
and hence each $\pi_i$ 
is a $W$-module homomorphism.
%It will be shown below that $\pi_i$ restricts to
%a bijection $\Phi_i(\mathscr{C'}) \leftrightarrow \Phi_i(\mathscr{C})$. 
 
\begin{proposition}
\label{pp:eqv}
For each $i\in\{1,2\}$, the restriction of $\pi_i$ defines a
$W$-equivariant bijection $\Phi_i(\mathscr{C'})\to\Phi_i(\mathscr{C})$. 
\end{proposition}

To prove Proposition \ref{pp:eqv}, we need a few elementary results and some
further notation.

\begin{definition}
For each $i \in \{1, 2\}$, define an equivalence relation $\sim_i$ on $\Phi_i(\mathscr{C})$
as follows:
if $z_1$ and $z_2 \in \Phi_i(\mathscr{C})$, then $z_1 \sim_i z_2$ if and only if $z_1$ and $z_2$
are (nonzero) scalar multiples
of each other.  For each $z \in \Phi_i(\mathscr{C})$, write $\widehat{z}$ for the equivalence
class
containing $z$, and set $\widehat{\Phi_i(\mathscr{C})}= \{\, \widehat{z} \mid z\in \Phi_i(\mathscr{C})
\,\}$. 
\end{definition} 

Observe that the action of $W$ on $\Phi_i(\mathscr{C})$ (for $i=1,\,2$) gives rise to a
well-defined action of $W$ on $\widehat{\Phi_i(\mathscr{C})}$ satisfying $w
\widehat{z}=\widehat{wz}$ for all $w\in W$, and all $z \in \Phi_i(\mathscr{C})$. 

\begin{definition}
For $i\in \{1, 2\}$, and for each $w\in W$, define 
$$
N_i(w) = \{\, \widehat{\gamma} \mid \text{$\gamma \in \Phi^+_i(\mathscr{C})$ and
$w\gamma\in \Phi^-_i(\mathscr{C})$}\, \}. 
$$
\end{definition}

Note that for $w\in W$, the set $N_i(w)$ ($i=1, 2$) can be
alternatively characterized as 
$
\{\, \widehat{\gamma} \mid \text{$\gamma \in \Phi^-_i(\mathscr{C})$ and
$w\gamma\in \Phi^+_i(\mathscr{C})$}\, \} 
$.
%\noalign{\vspace{-5 pt}\hbox{and}\vspace{-5 pt}}
Hence $\widehat{z}\in N_i(w)$ if and only if precisely one element of the
set
$\{z, wz\}$ is in $\Phi^+_i(\mathscr{C})$. A mild generalization of the techniques used in
(\cite[\S 5.6]{HM}) then yields the following result:     

\begin{lemma}
\label{lem:simp}
{\rm (i)\quad} If $s\in S$ then $N_1(r_s)=\{\widehat{\alpha_s}\}$ and
$N_2(r_s)=\{\widehat{\beta_s}\}$.
 
\noindent\rlap{\rm (ii)}\qquad Let $w\in W\!$. Then $N_1(w)$ and $N_2(w)$ both
have cardinality~$\ell(w)$.

\noindent\rlap{\rm (iii)}\qquad Let $w_1$, $w_2\in W$ and let $\dotplus $ denote
set symmetric difference. Then 
$
N_i(w_1w_2)=w_2^{-1}N_i(w_1)\dotplus N_i(w_2) 
$ for each $i\in \{1, 2\}$.
\par
\noindent\rlap{\rm (iv)}\qquad Let $w_1$, $w_2\in W$. Then for each $i\in \{1,2\}$,
$$
\ell(w_1w_2)=\ell(w_1)+\ell(w_2)\text{ if and only if }N_i(w_2)\subseteq N_i(w_1 w_2).
$$
\qed
\end{lemma}

If $w_1\, , w_2\in W$ with $\ell(w_1w_2)=\ell(w_1)+\ell(w_2)$ then we call
$w_2$ a \emph{right hand segment} of $w_1 w_2$.
If $w=r_{s_1} r_{s_2}\cdots r_{s_l}$ (where $w\in W$ and $s_1, \cdots, s_l \in S$) with
$\ell(w)=l$, then the above lemma yields that 
\begin{align}
\label{eq:N1}
N_1(w)&=\{ \widehat{\alpha_{s_l}},\, r_{s_l}\widehat{\alpha_{s_{l-1}}},\,
r_{s_l}r_{s_{l-1}}\widehat{\alpha_{s_{l-2}}},\, \ldots,\,
r_{s_l}r_{s_{l-1}}\ldots r_{s_2} \widehat{\alpha_{s_1}}\} \\
\noalign{\hbox{and}}
\label{eq:N2}
N_2(w)&=\{ \widehat{\beta_{s_l}},\,
r_{s_l}\widehat{\beta_{s_{l-1}}},\,r_{s_l}r_{s_{l-1}}\widehat{\beta_{s_{l-2}}},\
, \ldots,\, r_{s_l}r_{s_{l-1}}\ldots r_{s_2} \widehat{\beta_{s_1}}\}. 
\end{align}

\begin{lemma} 
\label{lem:gr}
$W$ is finite if and only if $\widehat{\Phi_i(\mathscr{C})}$ is finite (for $i=1,2$).
%\qed
\end{lemma}
\begin{proof}
It is clear that the finiteness of $W$ implies the finiteness of $\widehat{\Phi_i(\mathscr{C})}$ (for $i=1, 2$).
Conversely, for each $i\in \{1, 2\}$, assume that $|\widehat{\Phi_i}|<\infty$, and  
define an equivalence relation $\approx_i$ on $\Phi_i$ as follows: 
for $z_1,\,z_2\in \Phi_i$, write $z_1\approx z_2$ if there is a positive $\lambda$ such that $z_1 =\lambda z_2$. 
We write $\widetilde{z}$ for the equivalence class containing $z\in \Phi_i$, and set 
$\widetilde{\Phi_i}:=\{\, \widetilde{z}\mid z\in \Phi_i\,\}$. Observe that 
$|\widetilde{\Phi_i}|=2|\widehat{\Phi_i}|<\infty$.
The action of $W$ on $\Phi_i$ naturally induces a well-defined action of $W$ 
on $\widetilde{\Phi_i}$ satisfying $w\widetilde{z}:=\widetilde{wz}$.
Now for each $w\in W$ define a map $\sigma_w\colon \widetilde{\Phi_i} \to \widetilde{\Phi_i}$ 
by $\sigma_w(\widetilde{z}):=\widetilde{wz}$ for all $\widetilde{z}\in \widetilde{\Phi_i}$. 
Then $\sigma_w$ is a permutation of $\widetilde{\Phi_i}$, and furthermore, $w\mapsto \sigma_w$ 
is a homomorphism $\sigma: W \to \Sym(\widetilde{\Phi_i})$ (the symmetric group on $\widetilde{\Phi_i}$). 
Now if $w$ is in the kernel of $\sigma$ then $w\widetilde{z} =\widetilde{z}$ for all $z\in \Phi_i$, 
and in particular, $w \widetilde{z}=\widetilde{z}$ for all $z\in \Pi_i$. 
But by (ii)~Lemma \ref{lem:key} this means that $\ell(wr_s)>\ell(w)$ for all $s\in S$, 
and therefore $w=1$. Thus $\sigma$ is injective, and hence
$|W|\leq |\Sym(\widetilde{\Phi_i})|=|\widetilde{\Phi_i}|\,!<\infty$,
as required.
\end{proof}

\begin{remark}
\label{rm:para} 
Let $K\subseteq S$. If we define $V_{1K}$ to be the subspace of $V_1$ spanned by
$\Pi_{1K}=\{\,\alpha_s\mid s\in K\,\}$ and $V_{2K}$ to be the subspace of $V_2$
spanned by $\Pi_{2K}=\{\,\beta_s\mid s\in K\,\}$, and let $\langle\ ,\ \rangle_K$
be the restriction of $\langle\ ,\ \rangle$ to $V_{1K}\times V_{2K}$, then clearly
$\mathscr{C}_K=(K,V_{1K},V_{2K},\Pi_{1K},\Pi_{2K},\langle\ ,\ \rangle_K)$ is a Coxeter datum
with parameters $\{\,m_{st}\mid s,t\in K\,\}$. Next we write $W_K=\langle\{\,r_s'\mid
s\in K\,\}\rangle$ for the corresponding abstract Coxeter group, and let 
$\eta\colon W_K\to W$ be the homomorphism defined by $r_s'\mapsto r_s$ for all 
$s\in K$. It follows immediately from the formulas for the actions of $W$ on $V_1$ 
and $W_K$ on $V_{1K}$ that $r_s'v=r_sv$ for all $s\in K$ and
$v\in V_{1K}$, and therefore $wv=\eta(w)v$ for all $w\in W_K$ and $v\in V_{1K}$.
Since the action of $W_{K}$ on $V_{1K}$ is faithful, it follows that $\eta$ is 
injective. Thus $W_K$ can be identified with the \emph{standard parabolic subgroup\/} 
of $W$ generated by the set $\{\,r_s \mid s\in K \,\}$. 
\end{remark}

\begin{definition}
\label{def:3.10}
Given $K\subseteq S$, define $\Phi_{1K} (\subseteq V_{1K})$ and 
$\Phi_{2K}(\subseteq V_{2K})$ to be the root
systems for $W_K$ corresponding to $\mathscr{C}_K$, and 
$\Phi_{1K}^+,\,\Phi_{2K}^+$ to be the corresponding sets of positive roots.
\end{definition}
In other words, we have 
$\Phi_{1K}=\{\,w\alpha_r\mid w\in W_K\text{ and }r\in K\,\}$,
and $\Phi_{1K}^+=\Phi_{1K}\cap\PLC(\Pi_{1K})$; similarly for $\Phi_{2K}$ and
$\Phi_{2K}^+$.

\begin{remark}
\label{rm:gr}
It is a particular case of Lemma \ref{lem:gr} that $W_K$ is finite if and only if either
$\widehat{\Phi_{1K}}$ or 
$\widehat{\Phi_{2K}}$ is finite.
\end{remark}

\begin{remark}
\label{rm:longest}
It follows easily from Lemma~\ref{lem:key}~(ii)
%and Lemma~\ref{lem:simp} 
that if $W_K$ is finite then there exists a (unique) longest
element $w_K\in W_K$ satisfying the condition 
$N_1(w_K)=\{\,\widehat \alpha\mid\alpha\in\Phi_{1K}\,\}=\widehat{\Phi_{1K}}$.
For present purposes we require this only in the special
 case when $K$ has cardinality $2$.
\end{remark}

\begin{notation}
For $r,\,s\in S$, if $m_{rs}<\infty$ then $\langle \{ r_r, r_s\}\rangle$ 
is  finite, and its longest element, denoted by $w_{\{r,s\}}$,
is $r_rr_sr_r\cdots=r_sr_rr_s\cdots$, 
where there are $m_{rs}$ alternating factors on each side.  
\end{notation}

%\begin{remark}
%\label{rm:longest}
%Suppose that $s, t\in S$ such that $w_{\{r,s\}}$ exists. It follows from 
%Lemma~\ref{lem:key}~(ii) and the fact that $w_{\{r, s\}}$
%is the longest element in $W_{\{r, s\}}$ that $N_1(w_{\{r, s\}})=\widehat{\Phi_{1\, \{r, s\}}}$.
%\end{remark}

\begin{lemma}
\label{lem:1.24}
Suppose that $s, t\in S$ and $w\in W$  such that $w\alpha_s =\lambda \alpha_t$ for some
positive constant $\lambda$. Let $r\in S$ be such that $\ell(wr_r)<\ell(w)$. Then 
\begin{itemize}
\item [(i)] $\langle \{r_r, r_s \}\rangle$ is finite.
\item [(ii)] $N_1(w_{\{r,s\}} r_s) = \widehat{\Phi_{1\{\,r,s\,\}} }\setminus
\{\widehat{\alpha_s}\}$.
\item [(iii)] $\ell(w r_s w^{-1}_{\{r,s\}})= \ell(w)-\ell(w_{\{r,s\}}r_s)$.
\end{itemize}
\end{lemma}
\begin{proof}
(i)\quad Since $\ell(wr_r)<\ell(w)$, it follows from Lemma~\ref{lem:key}~(ii) that
$w\alpha_r \in \Phi^-_1(\mathscr{C})$. Furthermore,
$w\widehat{\alpha_r}\ne\widehat{\alpha_t}$, for otherwise $\widehat{\alpha_r} =\widehat{\alpha_s}$, 
contradicting  Lemma~\ref{lem:l1}. Observe that then
\begin{equation*}
 w r_s \alpha_r = w(\alpha_r -2\langle \alpha_r, \beta_s \rangle \alpha_s)
= \underbrace{w\alpha_r}_\text{$\in \Phi^-_1(\mathscr{C})\setminus\R\{\alpha_t\} $} -
\underbrace{ 2\lambda \langle \alpha_r, \beta_s \rangle \alpha_t}_\text{a scalar
multiple of $\alpha_t$ }.
\end{equation*}
Thus it can be checked that $w r_s\alpha_r \in \Phi_1^-(\mathscr{C})$. Now the fact that 
both $w r_s\alpha_r$ and $w r_s\alpha_s$ are negative 
implies that $w r_s(\lambda' \alpha_r + \mu' \alpha_s)$ is a negative linear 
combination of $\Pi_1$ whenever $\lambda',\,\mu'\ge0$. This says precisely that 
$\widehat{\Phi_{1\{\,r,\,s\,\}}} \subseteq N_1(w r_s)$. 
Since $N_1(w r_s)$ is a finite set of size $\ell(w r_s)$ by Lemma \ref{lem:simp} (ii), 
it follows from Remark~\ref{rm:gr} above that $\langle\{ r_r, r_s \}\rangle $ must be 
finite.

(ii)\quad
First, $w_{\{r,s\}}$ exists by part (i). Next let 
$\mu \alpha_r+\nu\alpha_s \in \Phi_1^+(\mathscr{C})$ (where $\mu, \nu \geq 0$) be arbitrary. 
Then it follows from Remark~\ref{rm:longest} that  
$w_{\{r, s\}} r_s (\mu\alpha_r+\nu\alpha_s)\in \Phi_1^+(\mathscr{C})$ if and only if  
$r_s(\mu\alpha_r+\nu\alpha_s) \in \Phi_1^-(\mathscr{C})$, which by Lemma~\ref{lem:simp}~(i)  happens if 
and only if $\mu=0$, and consequently 
$N_1(w_{\{r, s\}}r_s)=\widehat{\Phi_{1\,\{r, s\}}}\setminus \{\widehat{\alpha_s}\}$, as required.

(iii)\quad 
By Lemma \ref{lem:simp} (iv), to show that $w_{\{r,s\}} r_s$ is a right hand
segment of $w$, it is enough to show that
$N_1(w_{\{r,s\}} r_s) \subseteq N_1(w)$. Suppose that $\alpha\in \Phi^+_1(\mathscr{C})$ 
such that $\widehat\alpha\in N_1(w_{\{r,s\}} r_s)$. 
By (ii) above we may write $\alpha=c_1 \alpha_r+ c_2 \alpha_s$ for some $c_1>0$
and $c_2\geq 0$. Then
$
w\alpha = c_1 w\alpha_r +c_2\lambda\alpha_t.
$
Suppose for a contradiction that $w\alpha\in \Phi_1^+(\mathscr{C})$. Then
$c_2\lambda \alpha_t=w\alpha - c_1w\alpha_r \in \PLC(\Pi_1)$.
Rearranging this equation gives $\lambda'\alpha_t =\sum_{t'\in
S\setminus\{t\}}\lambda_{t'}\alpha_{t'}$, 
where $\lambda'$ is a constant and $\lambda_{t'}\geq 0$ for all $t'\in
S\setminus\{t\}$. Now if $\lambda'>0$ 
then we have a contradiction to Lemma~\ref{lem:l1}; on the other hand if
$\lambda'\leq 0$ then we have a contradiction 
to $0\notin \PLC(\Pi_1)$. 
\end{proof}

Now we are ready to prove Proposition \ref{pp:eqv}:
\begin{proof}[Proof of Proposition 
\ref{pp:eqv}]
Since $\Phi_1(\mathscr{C'})=\{\,w\alpha_s'\mid w\in W\text{ and }s\in S\,\}$, to prove that
the restriction of $\pi_1$ to $\Phi_1(\mathscr{C'})$ is bijective it suffices to show that
if $\pi_1(w\alpha_s')=\pi_1(v\alpha_t')$ for some $w,\,v\in W$ and $s,\,t\in S$,
then $w\alpha_s'=v\alpha_t'$. 
%Since $\pi_1$ is a $W$-homomorphism, and $\pi_1(\alpha_s')=\alpha_s$ for all $s\in S$, 
Observe that $\pi_1(w\alpha_s')=w\alpha_s$,
and $\pi_1(v\alpha_t')=v\alpha_t$. Hence it suffices to prove the following
statement: if $w\alpha_s=\lambda\alpha_t$ for some $w\in W$, $s,\,t\in S$, and $\lambda\neq 0$,
then $w\alpha_s'=\lambda\alpha_t'$.
We assume that $w \alpha_s = \lambda\alpha_t$,
and proceed by an induction on $\ell(w)$. The case $\ell(w)=0$ reduces to the
statement:
if $\alpha_s=\lambda\alpha_t$ for some $s,\,t\in S$ then $\alpha_s'=\lambda\alpha_t'$. 
Given $\alpha_s=\lambda\alpha_t$, Lemma~\ref{lem:l1} and the 
requirement that $0\notin\PLC(\Pi_1)$ together yield that $\lambda=1$ and $s=t$, 
and we are done.
Thus we may assume that
$\ell(w)>0$, and choose $r\in S$ such that $\ell(wr_r) <\ell(w)$. Lemma
\ref{lem:1.24} yields that
$\langle \{r_r, r_s\} \rangle$ is a finite dihedral group (hence $m_{rs}$ is
finite),  and $\ell(w(w_{\{r, s\}} r_s)^{-1})=\ell(w)-\ell(w_{\{r,s\}}r_s)$.
We treat separately the cases $m_{rs}$ even and $m_{rs}$ odd.

If $m_{rs}=2k$ is even, then
$w_{\{r,s\}} =(r_r r_s)^k=(r_s r_r)^k$, and then the formulas in Lemma 
\ref{lem:storder}~(i)
yield
\begin{align}
\label{eq:A}
(w_{\{r,s\}} r_s)\alpha_s&=  -w_{\{r,s\}}\alpha_s =-(r_s r_r)^k\alpha_s\notag\\
&= -\frac{\sin((m_{rs}+1)\pi/m_{rs})}{\sin(\pi/m_{rs})}
\alpha_s-\frac{-\cos(\pi/m_{rs})}{\langle \alpha_r, \beta_s
\rangle}\frac{\sin(\pi)}{\sin(\pi/m_{rs})} \alpha_r\notag\\
&= \alpha_s; 
\end{align}
and by exactly the same calculation in $V_1'$ we have:
\begin{equation}\label{eq:AA}
(w_{\{r,s\}} r_s)\alpha'_s= \alpha'_s. 
\end{equation}
Observe that (\ref{eq:A}) yields that
\begin{equation}
\label{eq:e1}
\alpha_t = w\alpha_s =w (w_{\{r,s\}} r_s)^{-1}(w_{\{r,s\}}r_s)\alpha_s
=w (w_{\{r,s\}} r_s)^{-1}\alpha_s. 
\end{equation}  
Now since $\ell(w (w_{\{r,s\}} r_s)^{-1})<\ell(w)$, 
the inductive hypothesis combined with (\ref{eq:e1}) give us
\begin{equation}
\label{eq:t'}
\alpha'_t = w(w_{\{r,s\}} r_s)^{-1}\alpha'_s.
\end{equation}
Then it follows from (\ref{eq:AA}) and (\ref{eq:t'}) that 
$$
w\alpha'_s=(w (w_{\{r,s\}}
r_s)^{-1}(w_{\{r,s\}}r_s))\alpha'_s=w (w_{\{r,s\}}
r_s)^{-1}\alpha'_s=\alpha'_t, 
$$
and the desired result follows by induction.

Next if $m_{rs}=2k+1$ is odd, then 
$w_{\{r,s\}}r_s =\underbrace{\ldots r_rr_s r_r}_\text{$(m_{rs}-1)$ factors}=(r_s r_r)^{k}$.
Then the formulas in Lemma \ref{lem:storder}~(i) yield that 
\begin{align}
\label{eq:AAA}
(w_{\{r,s\}}r_s) \alpha_s &= (r_s r_r)^{k}\alpha_s\notag\\
&=\frac{\sin
(\pi)}{\sin(\pi/m_{rs})}\alpha_s+\frac{-\cos(\pi/m_{rs})}{\langle \alpha_r,
\beta_s\rangle}\frac{\sin(2k\pi/m_{rs})}{\sin( \pi/m_{rs})} \alpha_r\notag\\
&=\frac{-\cos(\pi/m_{rs})}{\langle \alpha_r, \beta_s\rangle}\alpha_r;
\end{align}
and by exactly the same calculation in $V_1'$ we have: 
\begin{equation}
\label{eq:BBB}
(w_{\{r,s\}}r_s)\alpha'_s =\frac{-\cos(\pi/m_{rs})}{\langle \alpha_r, \beta_s\rangle}\alpha'_r.
\end{equation}
Observe that (\ref{eq:AAA}) yields that 
\begin{equation}
\label{eq:DDD}
%\begin{split}
\alpha_t = w(w_{\{r,s\}}r_s)^{-1}(w_{\{r,s\}}r_s) \alpha_s
= w (w_{\{r,s\}}r_s)^{-1}(\frac{-\cos(\pi/m_{rs})}{\langle \alpha_r,
\beta_s\rangle}\alpha_r). 
%\end{split}
\end{equation}
Since $\ell(w (w_{\{r,s\}}r_s)^{-1})<\ell(w)$, combining (\ref{eq:DDD}) 
and the inductive hypothesis yield that
\begin{equation}
\label{eq:EEE}
\alpha'_t=w (w_{\{r,s\}}r_s)^{-1}(\frac{-\cos(\pi/m_{rs})}{\langle \alpha_r,
\beta_s\rangle}\alpha_r').
\end{equation} 
Finally it follows from (\ref{eq:BBB}) and (\ref{eq:EEE}) that 
\begin{align*}
w\alpha'_s  = w (w_{\{r,s\}}r_s)^{-1} (w_{\{r,s\}}r_s)\alpha'_s
&=w (w_{\{r,s\}}r_s)^{-1}(\frac{-\cos(\pi/m_{rs})}{\langle \alpha_r,
\beta_s\rangle}\alpha_r')\\
&=\alpha'_t,
\end{align*}
completing the proof that $\pi_1$ restricts to a bijection 
$\Phi_1(\mathscr{C'})\leftrightarrow \Phi_1(\mathscr{C})$. 
Exactly the same reasoning also yields that $\pi_2$ restricts to a bijection 
$\Phi_2(\mathscr{C'})\leftrightarrow \Phi_2(\mathscr{C})$.
\end{proof}
Under the freeness condition  of $\mathscr{C'}$, each root 
$\alpha'\in \Phi_1(\mathscr{C'})$  can be written uniquely in the form 
$\alpha'=\sum_{s\in S} \lambda_s \alpha'_s$; 
we say that $\lambda_s$ is the \emph{coefficient} of $\alpha'_s$ in $\alpha'$, and denote it 
by $\coeff_{s}(\alpha')$. Similarly each root 
$\beta'\in \Phi_2(\mathscr{C'})$ can be written uniquely as 
$\beta'=\sum_{s\in S} \mu_s \beta'_s$; we say that $\mu_s$ is the \emph{coefficient} of $\beta'_s$ 
in $\beta'$, and denote it by $\coeff_{s} (\beta')$.  

\begin{definition}
(i)\quad Let $\alpha \in \Phi_1(\mathscr{C})$ be arbitrary. For each $s\in S$, define the
\emph{canonical coefficient} of $\alpha_s$ in the root $\alpha$, written 
$\coeff_{s}(\alpha)$, by requiring that $\coeff_{s}(\alpha)=\coeff_s(\pi^{-1}_1(\alpha))$, 
where $\pi_1$ is as in Proposition~\ref{pp:eqv}. The \emph{support} of the root
$\alpha$, written $\supp(\alpha)$, is the set $\{\,\alpha_s\mid \coeff_{s}(\alpha)\neq 0\,\}$.

\noindent\rm (ii)\quad  Let $\beta \in \Phi_2(\mathscr{C})$ be arbitrary. For each $s\in S$,
define the \emph{canonical coefficient} of $\beta_s$ in the root $\beta$, written
$\coeff_{s}(\beta)$, by requiring that $\coeff_{s}(\beta)~=~\coeff_s(\pi^{-1}_2(\beta))$, 
where $\pi_2$ is as in Proposition~\ref{pp:eqv}. The \emph{support} of the root
$\beta$, written $\supp(\beta)$, is the set $\{\,\beta_s\mid \coeff_{s}(\beta)\neq 0\,\}$.

\end{definition}

A similar proof as that of Proposition \ref{pp:eqv} establishes the following:
\begin{proposition}
\label{pp:phi}
Suppose that $s, t \in S$. If $w\alpha_s=\lambda\alpha_t$ for some $w\in W$
and some non-zero constant $\lambda$,
then $w\beta_s = \tfrac{1}{\lambda}\beta_t$.
\qed
\end{proposition}

In particular, Proposition \ref{pp:phi} implies the following: if 
$w_1 \alpha_s= w_2 \alpha_t$ for some $s, t\in S$ and $w_1, w_2\in W$,
then $w_1 \beta_s= w_2\beta_t$. Thus  there exists a well-defined map 
$\Phi_1(\mathscr{C})\to\Phi_2(\mathscr{C})$ satisfying the requirement that 
$w\alpha_s\mapsto w\beta_s$ for all $s\in S$ and $w\in W$. This is clearly
the unique $W$-equivariant map $\Phi_1(\mathscr{C})\to\Phi_2(\mathscr{C})$ satisfying 
$\alpha_s\mapsto \beta_s$ for all $s\in S$. Furthermore, by going through a similar 
reasoning as in the proof of Proposition \ref{pp:eqv}, we may conclude 
that this map is a bijection. Summing up, we have:
 
\begin{proposition}
\label{pp:phimap}
There exists a unique $W$-equivariant bijection $\phi: \Phi_1(\mathscr{C}) \to \Phi_2(\mathscr{C})$  
satisfying $\phi(\alpha_s) =\beta_s$ for all $s\in S$.
\qed
\end{proposition}

For the rest of this paper, the notation $\phi$ will be fixed for the $W$-equivariant bijection in the above 
proposition.
Observe that Remark~\ref{rm:sign} then immediately implies the following:
\begin{corollary}
\label{co:bijection}
Let $\alpha\in \Phi_1(\mathscr{C})$. Then $\alpha\in \Phi_1^+(\mathscr{C})$ if and only if
$\phi(\alpha)\in \Phi_2^+(\mathscr{C})$; and similarly $\alpha\in \Phi_1^-(\mathscr{C})$ 
if and only if $\phi(\alpha)\in \Phi_2^-(\mathscr{C})$.
\qed 
\end{corollary}

Note that Proposition \ref{pp:phi} can be generalized to the following:
\begin{lemma}
\label{lem:inv}
Suppose that $\alpha \in \Phi_1(\mathscr{C})$, and suppose that $\lambda$ is a non-zero
constant such that $\lambda \alpha \in \Phi_1(\mathscr{C})$. Then 
$\phi(\lambda \alpha)=\tfrac{1}{\lambda}\phi(\alpha)$.
\end{lemma}
\begin{proof}
Write $\alpha= w \alpha_s$ for some $w\in W$ and $s\in S$. The fact that
$\lambda \alpha\in \Phi_1(\mathscr{C})$ implies that
$w^{-1}(\lambda \alpha)=\lambda w^{-1}\alpha =\lambda \alpha_s \in \Phi_1(\mathscr{C})$.
Then $\phi(\lambda \alpha_s) =\frac{1}{\lambda} \phi(\alpha_s)$ by 
Proposition~\ref{pp:phi}. Now it follows from
the $W$-equivariance of $\phi$ that
$$
\phi(\lambda \alpha)=\phi(\lambda w\alpha_s)=\phi(w \lambda \alpha_s)
=w \phi(\lambda \alpha_s) =\frac{1}{\lambda}
w\phi(\alpha_s)=\frac{1}{\lambda}\phi(\alpha).
$$
\end{proof} 

\begin{definition}
\label{df:dp} Let $\mathscr{C}=(\, S, V_1, V_2, \Pi_1, \Pi_2, \langle\,,\,\rangle \,)$ 
be a Coxeter datum. 
\noindent(i)\quad For each $i\in\{1,2\}$ and each  $z \in \Phi^+_i(\mathscr{C})$, define the
\emph{depth} of $z$ (written $\dep_{\mathscr{C}, i}(z)$ ) to be 
$$\dep_{\mathscr{C},i}(z)=\min \{\,\ell(w) \mid \text{ $w\in W$ and $w z\in \Phi^-_i(\mathscr{C})$}\,\}.$$
%\noindent(i)\quad For each $\alpha \in \Phi^+_1(\mathscr{C})$  and $\beta \in \Phi^+_2(\mathscr{C})$ define the
%\emph{depth} of $\alpha$ (written $\dep_{\mathscr{C}, 1}(\alpha)$ ) and the \emph{depth} of
%$\beta$ (written $\dep_{\mathscr{C},2}(\beta)$) to be 
%\begin{align*}
%&\dep_{\mathscr{C},1}(\alpha)=\min \{\,\ell(w) \mid \text{ $w\in W$ and $w\alpha\in \Phi^-_1(\mathscr{C})$}
%\,\},\\
%\noalign{\hbox{and}}
%&\dep_{\mathscr{C},2}(\beta)=\min \{\,\ell(w) \mid \text{ $w\in W$ and $w\beta \in\Phi^-_2(\mathscr{C})$}
%\,\}.
%\end{align*}
(ii)\quad For each $i\in\{1,2\}$ and $z_1, z_2 \in \Phi^+_i(\mathscr{C})$,
write $z_1 \preceq_i z_2$ if  there exists $w\in W$  such that
$z_2 = w z_1$, and $\dep_{\mathscr{C},i}(z_2)=\dep_{\mathscr{C},i}(z_1)+\ell(w)$.
Furthermore, we write $z_1 \prec_i z_2$ if
$z_1 \preceq_i z_2$ but $z_1 \neq z_2$. 
%(ii)\quad For $\alpha_1, \alpha_2 \in \Phi^+_1(\mathscr{C})$ (respectively $\beta_1, \beta_2
%\in \Phi^+_2(\mathscr{C})$), write $\alpha_1 \preceq_1 \alpha_2$ (respectively $\beta_1 \preceq_2
%\beta_2$) if and only if there exists $w\in W$ (respectively $w'\in W$) such that
%$\alpha_2 = w \alpha_1$, and $\dep_{\mathscr{C},1}(\alpha_2)=\dep_{\mathscr{C},1}(\alpha_1)+\ell(w)$
%(respectively $\beta_2 = w'\beta_1$, and $\dep_{\mathscr{C},2}(\beta_2)=
%\dep_{\mathscr{C},2}(\beta_1)+\ell(w')$). Furthermore, we write $\alpha_1 \prec_1 \alpha_2$ if
%$\alpha_1 \preceq_1 \alpha_2$ but $\alpha_1 \neq \alpha_2$ (respectively
%$\beta_1 \prec_2 \beta_2$ if $\beta_1 \preceq_2 \beta_2$ but $\beta_1 \neq
%\beta_2$).
\end{definition}

A mild generalization of \cite[Lemma 1.7]{BH93} yields the following:
\begin{lemma}
\label{lem:1.33}
Suppose that $\mathscr{C}=(\,S, V_1, V_2, \Pi_1, \Pi_2, \langle\,,\,\rangle\,)$ is a Coxeter datum.
Let $s\in S$, $\alpha \in \Phi^+_1(\mathscr{C})\setminus \R\{\alpha_s\}$, and $\beta \in
\Phi^+_2(\mathscr{C})\setminus \R\{\beta_s\}$. Then
\begin{equation*}
\dep_{\mathscr{C},1}(r_s \alpha) =
\begin{cases}
\dep_{\mathscr{C},1}(\alpha)-1  \text{   if $\langle \alpha, \beta_s \rangle  >0$} ,\\
\dep_{\mathscr{C},1}(\alpha) \ \ \ \ \ \text{   if $\langle \alpha, \beta_s \rangle = 0$}, \\
\dep_{\mathscr{C},1}(\alpha)+1  \text{   if $\langle \alpha, \beta_s \rangle <0$}; 
\end{cases}
\end{equation*}
and
\begin{equation*}
\dep_{\mathscr{C},2}(r_s \beta) =
\begin{cases}
\dep_{\mathscr{C},2}(\beta)-1  \text{   if $\langle \alpha_s, \beta \rangle  >0$} ,\\
\dep_{\mathscr{C},2}(\beta) \ \ \ \ \ \text{   if $\langle \alpha_s, \beta \rangle = 0$}, \\
\dep_{\mathscr{C},2}(\beta)+1  \text{   if $\langle \alpha_s, \beta \rangle <0$}. 
\end{cases}
\end{equation*}
\qed
\end{lemma}

\begin{lemma}
\label{lem:depeq}
Suppose that $\mathscr{C}=(\,S, V_1, V_2, \Pi_1, \Pi_2, \langle\,,\,\rangle\,)$ is a Coxeter datum, and
$\alpha \in \Phi_1^+(\mathscr{C})$. Then $\dep_{\mathscr{C},1}(\alpha)=\dep_{\mathscr{C},2}(\phi(\alpha))$.
\end{lemma}
\begin{proof}
Let $w\in W$ be such that $w\alpha \in \Phi^-_1(\mathscr{C})$ and 
$\dep_{\mathscr{C},1}(\alpha)=\ell(w)$. Then Corollary \ref{co:bijection} yields that 
$\phi(w\alpha)\in \Phi_2^-(\mathscr{C})$, and it follows from the $W$-equivariance of $\phi$ that, 
$w(\phi(\alpha))\in \Phi_2^-(\mathscr{C})$. Hence
$$\dep_{\mathscr{C},2}(\phi(\alpha))\leq \ell(w)= \dep_{\mathscr{C},1}(\alpha).$$ By symmetry, 
$\dep_{\mathscr{C},1}(\alpha) \leq \dep_{\mathscr{C},2}(\phi(\alpha))$, whence equality.  
\end{proof}

\begin{lemma}
\label{lem:1.35}
Suppose that $\mathscr{C}=(\,S, V_1, V_2, \Pi_1, \Pi_2, \langle\,,\,\rangle\,)$ is a Coxeter datum. Let
$\alpha \in \Phi_1(\mathscr{C})$ and $s\in S$. Then 
\begin{align*}
&\langle \alpha , \beta_s \rangle >0 \text{ if and only if } \langle \alpha_s, \phi(\alpha) \rangle >0 ; \\
\noalign{\hbox{and}}
&\langle \alpha , \beta_s \rangle =0 \text{ if and only if } \langle \alpha_s, \phi(\alpha) \rangle =0 ; \\
\noalign{\hbox{and}}
&\langle \alpha , \beta_s \rangle <0 \text{ if and only if } \langle \alpha_s, \phi(\alpha) \rangle <0 .
\end{align*}
\end{lemma} 
\begin{proof}
Combine Lemma \ref{lem:1.33} and Lemma \ref{lem:depeq}, then the desired result follows.
\end{proof}

In fact, the above gives rise to a more general result:

\begin{corollary}
\label{co:sign1}
Suppose that $\mathscr{C}=(\,S, V_1, V_2, \Pi_1, \Pi_2, \langle\,,\,\rangle\,)$ is a Coxeter datum, and
 $\alpha_1, \alpha_2 \in \Phi_1(\mathscr{C})$. Then
\begin{align*}
\langle \alpha_1, \phi(\alpha_2)\rangle >0 \quad \text{if and only if} \quad \langle \alpha_2, \phi(\alpha_1)\rangle >0;\\
\noalign{\hbox{and}}
\langle \alpha_1, \phi(\alpha_2)\rangle =0 \quad \text{if and only if} \quad \langle \alpha_2, \phi(\alpha_1)\rangle =0;\\
\noalign{\hbox{and}}
\langle \alpha_1, \phi(\alpha_2)\rangle <0 \quad \text{if and only if} \quad \langle \alpha_2, \phi(\alpha_1)\rangle <0.
\end{align*}
\end{corollary}
\begin{proof}
Write $\alpha_2 = w\alpha_s $ for some $w\in W$ and $s\in S$. Given the 
 $W$-invariance of $\langle\, ,\,\rangle$  and the  $W$-equivariance of $\phi$, we have
$$
\langle \alpha_1, \phi(\alpha_2)\rangle =\langle \alpha_1, w\beta_s\rangle
=\langle w^{-1}\alpha_1, \beta_s\rangle.
$$
It then follows from Lemma~\ref{lem:1.35} that $\langle w^{-1}\alpha_1,
\beta_s\rangle>0$ precisely when 
%$\langle \alpha_s,
%\phi(w^{-1}\alpha_1)\rangle>0$, and the latter happens if and only if
%$\langle w\alpha_s,
%\phi(\alpha_1)\rangle =\langle \alpha_2, \phi(\alpha_1)\rangle>0$
$$
\langle \alpha_s, \phi(w^{-1}\alpha_1)\rangle=\langle w\alpha_s,
\phi(\alpha_1)\rangle =\langle \alpha_2, \phi(\alpha_1)\rangle>0.
$$
The rest of the desired result follows in a similar way.
\end{proof}

\section[Comparison With Standard Geometric Realization]
{Comparison with the Standard Geometric Realization of Coxeter Groups}
\label{sec:cwsgr}
In this section we recover the root systems of Coxeter groups in the sense of
\cite{HH81} or \cite{HM} as special cases of root systems arising from a
Coxeter datum. 
Subsequently we give comparison results between such special cases and the more general
root systems arising from a Coxeter datum. These comparisons will 
provide useful reduction of the non-orthogonal representations studied in 
the previous sections into those of \cite[Ch.V]{HH81} or \cite[\S 5.3--5.4]{HM}.

Fix a Coxeter datum $\mathscr{C}=(\, S, V_1, V_2, \Pi_1, \Pi_2,
\langle\,,\,\rangle\,)$, and let $V$ be a vector space over $\R$ with a basis 
$\Pi= \{\,\gamma_s \mid s\in S \,\}$ in bijective correspondence with $S$.
Suppose that $(\, ,\, ) \colon  V\times V \to \R$ is a symmetric bilinear form
satisfying the following conditions:
\begin{itemize}
\item [(C1')] $(\gamma_s, \gamma_s) = 1$, for all $s\in S$;
\item [(C2')] $(\gamma_s, \gamma_t)\leq 0$, for all distinct $s, t \in S$;
\item [(C3')] $(\gamma_s, \gamma_t)^2 =\langle \alpha_s, \beta_t \rangle \langle
\alpha_t, \beta_s \rangle,$ for all $s, t \in S$.
\end{itemize}
Then $\mathscr{C}''=(\,S, V, V,\Pi, \Pi, (\,,\,)\,)$ is
a free Coxeter datum with the same Coxeter parameters $m_{st}$, ($s, t\in S$) 
as those of $\mathscr{C}$, and hence 
$\mathscr{C}$ and $\mathscr{C''}$  are associated to
the same abstract Coxeter system $(W, R)$. It is easily checked that
$\Phi_1(\mathscr{C''})~=~\Phi_2(\mathscr{C''})$, allowing us to write $\Phi$ 
in places of $\Phi_1(\mathscr{C''})$ and $\Phi_2(\mathscr{C''})$. Furthermore, 
we write $\Phi^+$ and $\Phi^-$ for the corresponding set of positive roots 
and negative roots respectively. It is also readily checked that $W$ can be 
faithfully embedded into the orthogonal group of the bilinear form $(\,,\,)$ on $V$.
This recovers the orthogonal representation of $W$ as defined in 
\cite[Ch.V]{BN68}, \cite{HH81} and \cite[\S 5.3--5.4]{HM}.
We refer to such $V$ as the \emph{standard geometric realization (Tits
representation)} of $W$ or simply the \emph{standard Tits
representation} of $W$.

It follows from Lemma \ref{lem:eqf} that $(\,,\,)$ is $W$-invariant. 
It is well-known (and can be readily checked) that 
if $x$ and $\lambda x$ are both in $\Phi$ for some constant $\lambda$, 
then $\lambda=\pm 1$. Since $\mathscr{C''}$ is free,  it follows that each  
$\gamma\in \Phi$  can be written uniquely in the form 
$\gamma=\sum_{s\in S} \lambda_s \gamma_s$.
We say that $\lambda_s$ is the \emph{coefficient} of $\gamma_s$ in $\gamma$, and denote it 
by $\coeff_{s}(\gamma)$.
%Further, since $\mathscr{C''}$ is free, it follows 
%that the canonical expressions of roots in $\Phi$ in terms
%of elements of $\Pi$ is the only possible expression. 
Finally, to simplify
notation, for any $\gamma \in \Phi$ we write $\dep (\gamma)$ in place of
$\dep_{\mathscr{C''}, 1}(\gamma)(=\dep_{\mathscr{C''}, 2}(\gamma))$, and for
$\gamma_1, \gamma_2 \in \Phi$ we simply write $\gamma_1 \preceq \gamma_2$ when $\gamma_1\preceq_1 \gamma_2$ 
(and $\gamma_1 \preceq_2\gamma_2$).

Similar arguments as for Proposition \ref{pp:eqv} gives us the following:
\begin{proposition}
\label{pp:fs}
There are $W$-equivariant maps $\phi_1 \colon \Phi_1(\mathscr{C}) \to \Phi$, and 
$\phi_2 \colon \Phi_2(\mathscr{C}) \to \Phi$, satisfying
$\phi_1(\alpha_s)= \gamma_s = \phi_2(\beta_s)$,
for all $s\in S$.
\qed
\end{proposition}

\begin{remark}
We stress that unlike $\pi_1$, $\pi_2$ of Proposition~\ref{pp:eqv} 
and $\phi$ of Proposition~\ref{pp:phimap}, the new maps
$\phi_1$ and $\phi_2$ are not injective in general, and we shall 
see more on this fact in Lemma~\ref{lem:coeff0}~(ii) below.
\end{remark}

\begin{lemma}
\label{lem:sign}
Suppose that $\alpha \in  \Phi_1(\mathscr{C})$. Then $\alpha \in
\Phi^+_1(\mathscr{C})$ implies that $\phi_1(\alpha)
\in \Phi^+$, and $\alpha \in \Phi^-_1(\mathscr{C})$ implies that 
$\phi_1(\alpha)\in \Phi^-$.
\end{lemma} 
\begin{proof}
We may write $\alpha = w \alpha_r$, for suitable
$w\in W$ and $r\in S$. If $\alpha \in \Phi^+_1(\mathscr{C})$, then
Lemma~\ref{lem:key}~(ii) yields that $\ell(w r_r) = \ell(w)+1$, in which case
Lemma \ref{lem:key}~(ii) applied to the Coxeter datum $\mathscr{C''}$ yields
that 
$\phi_1(\alpha)=\phi_1(w \alpha_r) = w
\phi_1(\alpha_r)= w \gamma_r \in \Phi^+$.
Likewise we see that $\alpha \in \Phi^-_1(\mathscr{C})$ implies that
$\phi_1(\alpha) \in\Phi^-$. 
\end{proof}

Using the same argument as in Lemma \ref{lem:depeq} we have:
\begin{lemma}
\label{lem:dpinv}
Suppose that $\alpha \in \Phi^+_1(\mathscr{C})$. Then 
$\dep_{\mathscr{C},1}(\alpha) = \dep(\phi_1(\alpha))$.
\qed
\end{lemma}

\begin{corollary}
\label{co:sign}
Suppose that $\alpha \in \Phi_1(\mathscr{C})$, and $s\in S$. Then 
\begin{align*}
&(\phi_1(\alpha), \gamma_s) > 0 \text{ if and only if } \langle \alpha, \beta_s \rangle > 0,\\
\noalign{\hbox{and}}
&(\phi_1(\alpha), \gamma_s) = 0 \text{ if and only if } \langle \alpha, \beta_s \rangle = 0,\\
\noalign{\hbox{and}}
&(\phi_1(\alpha), \gamma_s) < 0 \text{ if and only if } \langle \alpha, \beta_s \rangle < 0.
\end{align*}
\end{corollary}

\begin{proof}
Follows from Lemma \ref{lem:dpinv} and Lemma \ref{lem:1.33} applied to $\mathscr{C''}$.
\end{proof}

Next we have a well-known result on Coxeter groups (a proof can be found in
\cite{BBHT}, in the discussion immediately before Lemma 2.1) 
which we shall use repeatedly in later calculations.

\begin{lemma}
\label{lem:minleng}
Suppose that $I\subseteq S$ and $w\in W$. Let $W_I$ denote the standard 
parabolic subgroup in $W$ corresponding to $I$, as defined in Remark~\ref{rm:para}. 
Choose $w' \in w W_I$ to be of minimal length in the left coset 
of $W_{I}$ in $W$ containing $w$. Then
$\ell(w'v)=\ell(w')+\ell(v)$ for all $v \in W_I$. 
\qed
\end{lemma}

\begin{proposition}
\label{pp:1.44}
For each $\alpha \in \Phi_1(\mathscr{C})$ and each $r\in S$, 
$$\coeff_{r}(\alpha)\coeff_{r}(\phi(\alpha)) \geq
(\coeff_{r}(\phi_1(\alpha)))^2. $$
\end{proposition}
\begin{proof}
Replace $\alpha$ by $-\alpha$ if needed, we may assume that $\alpha \in
\Phi^+_1(\mathscr{C})$, and we may write $\alpha =w \alpha_s$, where $w\in W$
and
$s\in S$. The proof is based on  an induction on $\ell(w)$. If $\ell(w)=0$ then
the result clearly holds with an equality. Thus we may assume that 
$\ell(w) \geq 1$, and choose
$t\in S$ such that $\ell(wr_t)=\ell(w)-1$. Observe that $w=w_1 w_2$, where
$w_1$ is of minimal length in the coset $w \langle \{r_s, r_t \}\rangle$, and
$w_2$ is an alternating product of $r_s$ and $r_t$, ending in $r_t$. Then Lemma
\ref{lem:minleng} yields that $\ell(w)= \ell(w_1)+\ell(w_2)$, $\ell(w_1
r_s)=\ell(w_1)+1$, and $\ell(w_1 r_t )= \ell(w_1)+1$. Consequently  $w_1 \alpha_s\in
\Phi^+_1(\mathscr{C})$ and
$w_1 \alpha_t \in \Phi^+_1(\mathscr{C})$ by Lemma~\ref{lem:key}~(ii). 
The formulas in Lemma~\ref{lem:storder}
yield
that $w_2 \alpha_s = p\alpha_s+ \lambda q \alpha_t$, where $\lambda$ is a
positive constant and $pq \geq 0$. If $p$ and $q$ are both negative, then 
$$\alpha=w\alpha_s = w_1 w_2 \alpha_s = w_1(p\alpha_s +\lambda q\alpha_t)= p
w_1\alpha_s + \lambda q w_1\alpha_t\in \Phi_1^-(\mathscr{C}),$$
contradicting the assumption that $\alpha\in \Phi_1^+(\mathscr{C})$. Therefore
$p, q \geq 0$. By Lemma~\ref{lem:l1}, $\{\alpha_s, \alpha_t\}$ is
linearly independent, hence $\coeff_s(w_2\alpha_s)=p$, and
$\coeff_t(w_2 \alpha_s)=\lambda q$. Similar calculations as in Lemma~\ref{lem:storder}  yield that 
$$ w_2 \gamma_s = p \gamma_s + q \gamma_t,\quad\text{and}\quad w_2 \beta_s
= p \beta_s + \tfrac{q}{\lambda} \beta_t.$$
And Lemma ~\ref{lem:l1} implies that 
$\coeff_s(w_2 \beta_s)=p$ and $\coeff_t(w_2\beta_s)=\tfrac{q}{\lambda}$.
Next we set 
\begin{align*}
&x = \coeff_{r}(\alpha) \text{ , } x' = \coeff_{r}(\phi(\alpha)) \text{ , }
x'' = \coeff_{r}(\phi_1(\alpha));\\
\noalign{\hbox{and}}
&y = \coeff_{r}(w_1\alpha_s)\text{ , } y'  = \coeff_{r}(w_1 \beta_s) \text{ ,
} y'' = \coeff_{r}(w_1 \gamma_s);\\
\noalign{\hbox{and}}
&z = \coeff_{r}(w_1\alpha_t)\text{ , } z' = \coeff_{r}(w_1 \beta_t) \text{ , }
z'' = \coeff_{r}(w_1 \gamma_t).
\end{align*}
Then
\begin{align}
\label{eq:sum}
x x' - x''^2&= (p y+\lambda q z)(py'+\tfrac{1}{\lambda}q z')-(py''+qz'')^2\nonumber\\
&= p^2 (yy'-y''^2)+q^2 (zz'-z''^2)+pq (\tfrac{1}{\lambda}yz' + \lambda z y' - 2
y'' z'').
\end{align}
Since $\ell(w_1)<\ell(w)$, it follows from the inductive hypothesis that 
$$yy' \geq (y'')^2 \qquad\text{and}\qquad zz' \geq (z'')^2,
$$
and so the first two summands in the last line of (\ref{eq:sum}) are nonnegative.
Next apply the geometric mean and arithmetic mean
inequality to the terms $\tfrac{1}{\lambda}yz'$ and $\lambda y'z$ , and we can conclude that 
$\tfrac{1}{\lambda} yz'+\lambda y'z\geq 2\sqrt{yy' zz'}$. But the inductive hypothesis  yields 
that $ yy' zz' \geq (y'' z'')^{2}$, 
showing that the third summand in the last line of(\ref{eq:sum}) is also nonnegative, whence
$xx' -x''^2 \geq 0$, and the desired result follows by induction.
\end{proof}

\begin{lemma}
\label{lem:coeff0}
Suppose that $\alpha \in \Phi_1(\mathscr{C})$.

\noindent (i)\quad Let $t\in S$. Then 
$$\begin{cases}
\coeff_t(\alpha)=0 &\text{if and only if } \coeff_t(\phi_1(\alpha))=0,\\
\coeff_{t}(\alpha) > 0 &\text{if and only if } \coeff_{t} (\phi_1(\alpha)) > 0,\\
\coeff_{t}(\alpha) < 0 &\text{if and only if } \coeff_{t} (\phi_1(\alpha)) < 0 .
\end{cases}
$$
\noindent (ii) \quad The situation $\phi_1(\alpha) =\pm\gamma_s$ ($s\in S$) arises
if and only if $\alpha =\lambda \alpha_s$, for some nonzero constant $\lambda$.
\end{lemma}
\begin{proof}
(i)\quad Lemma \ref{lem:key} (i) applied to the Coxeter datum $\mathscr{C''}$ yields that
$\Phi=\Phi^+\uplus \Phi^-$. Hence we only need to verify that $\coeff_t(\alpha)=0$ if and only
if $\coeff_t(\phi_1(\alpha))=0$, for the rest of (i) follows readily from Lemma~\ref{lem:sign}. 
By Proposition \ref{pp:1.44} we only need to show that the condition
$\coeff_{t}(\phi_1(\alpha))=0$ implies that $\coeff_{t}(\alpha)=0$. 
Replacing $\alpha$ by $-\alpha$ if needs be, we may assume that $\alpha \in
\Phi^+_1(\mathscr{C})$, and write $\alpha =w \alpha_r$ for some $w\in W$ and $r\in S$. If
$\ell(w)=0$ then there is nothing to prove. Thus we may assume that $\ell(w)\geq
1$ and proceed by an induction on $\ell(w)$. Choose $s\in S$ such that
$\ell(wr_s) = \ell(w)-1$. Now  $w = w_1 w_2$, where $w_1\in w \langle\{ r_r,
r_s\} \rangle$ is of minimal length, and $w_2$ is an alternating product of $r_r
$ and $r_s$ ending in $r_s$. Then Lemma \ref{lem:minleng} implies that $w_2$ is a right 
segment of $w$. Since $\alpha=w \alpha_r$ is positive, it follows from 
Lemma~\ref{lem:simp}~(iv) that 
$w_2 \alpha_r$ is positive too.
Direct calculations as those in Lemma~\ref{lem:storder} show that
$$w_2 \gamma_r = p \gamma_r + q \gamma_s, \quad w_2 \alpha_r = p \alpha_r +
\lambda q \alpha_s \ \text{ and } \ w_2 \beta_r = p\beta_r
+\tfrac{q}{\lambda}\beta_s$$
for some non-negative constants $p$ and $q$, and some positive constant $\lambda$.
Now Lemma~\ref{lem:l1} yields that $\coeff_r(w_2 \alpha_r)=p$, $\coeff_s(w_2\alpha_r)=\lambda q$, 
$\coeff_r(w_2\beta_r) =p$ and $\coeff_s(w_2\beta_r)=\tfrac{q}{\lambda}$. 
Then 
\begin{align*}
0&=\coeff_{t}(\phi_1(\alpha))=\coeff_{t}(w\gamma_r)=\coeff_{t}
(w_1 (p\gamma_r+q\gamma_s))\\
 &=p\coeff_{t}(w_1\gamma_r)+q\coeff_{t}(w_1\gamma_s).
\end{align*}   
By Lemma \ref{lem:minleng} and Lemma \ref{lem:key} (ii) 
(applied to the Coxeter datum $\mathscr{C''}$), $w_1 \gamma_r$ and
$w_1 \gamma_s$ are both positive, and hence it follows from the above that 
$$p\coeff_{t}(w_1 \gamma_r) = q \coeff_{t}(w_1\gamma_s)=0.$$
Then the inductive hypothesis yields that 
$$p\coeff_{t}(w_1\alpha_r) = q \coeff_{t}(w_1 \alpha_s)=0,$$
and therefore 
$$
\coeff_{t}(\alpha)=\coeff_{t}(w\alpha_r)=p\coeff_{t}(w_1\alpha_r)+\lambda q \coeff_{t}(w_1\alpha_s)=0
$$
as required, and (i) follows by induction.

\noindent (ii)\quad Follows readily from part (i) above.
\end{proof}

\begin{remark}
\label{rem:also0}
Suppose that $\beta \in \Phi_2(\mathscr{C})$, and $t\in S$. Then the same 
arguments as those used in the proof of the above Lemma yield that 
$$\begin{cases}
\coeff_t(\beta)=0 &\text{ if and only if  } \coeff_t(\phi_2(\beta))=0,\\
\coeff_{t}(\beta) > 0 &\text{ if and only if  } \coeff_{t} (\phi_2(\beta)) > 0,\\
\coeff_{t}(\beta) < 0 &\text{ if and only if  } \coeff_{t} (\phi_2(\beta)) < 0 .
\end{cases}
$$
\end{remark}

The next result is taken from \cite{BB98}: 
\begin{lemma}\textup{(Brink \cite[Proposition 2.1]{BB98})}
\label{lem:ge1}
Suppose that $\gamma \in \Phi$ and $r\in S$.
Then $\coeff_{r}(\gamma) >0$ implies that $\coeff_{r}(\gamma)\geq
1$. 

Further, suppose that $0 < \coeff_{r} (\gamma) <2$. Then either
$\coeff_{r}(\gamma) =1$ or $\coeff_{r}(\gamma) = 2
\cos(\frac{\pi}{m_{r_1 r_2}})$, where $r_1$, $r_2\in S$ with $4 \leq m_{r_1 r_2}
< \infty$. 
\qed
\end{lemma}
Combining the results in Proposition~\ref{pp:1.44} to Lemma~\ref{lem:ge1} (inclusive) and 
an argument similar to the one used in the proof of Proposition~\ref{pp:1.44}, 
 we may deduce the following: 
\begin{proposition}
\label{pp:geq1}
Let $\alpha \in \Phi_1(\mathscr{C})$ and  $t\in S$ be arbitrary. Then
$$
\coeff_{t}(\alpha) > 0 \text{ if and only if } \coeff_{t}(\phi(\alpha)) >0,
$$ 
and in this case $\coeff_{t}(\alpha) \coeff_{t}(\phi(\alpha))\geq
1$. In particular, 
$$\phi(\supp(\alpha))=\supp (\phi(\alpha)).$$
Furthermore, suppose that  
$1\leq
\coeff_{t}(\alpha)\coeff_{t}(\phi(\alpha)) <4$.
Then either 
$$\coeff_{t}(\alpha)\coeff_{t}(\phi(\alpha)) =1, \text{ or else } 
\coeff_{t}(\alpha)\coeff_{t}(\phi(\alpha)) =4 \cos^2(\frac{\pi}{m}),$$
%\begin{align*}
%&\coeff_{r}(\alpha)\coeff_{r}(\phi(\alpha)) =1\\
%\noalign{\hbox{or else}}
%&\coeff_{r}(\alpha)\coeff_{r}(\phi(\alpha)) =4 \cos^2(\frac{\pi}{m}), 
%\end{align*}
where $m=m_{r_1 r_2}$, for some $r_1, r_2 \in S$ with $4\leq m < \infty$.
In particular, if $\alpha \in \Phi^+_1(\mathscr{C})$ then
$\coeff_{t} (\alpha) \coeff_{t}(\phi(\alpha)) =1$  if and only if
$\coeff_{t}(\phi_1(\alpha)) =1$.
\qed
\end{proposition}

\begin{proposition}
\label{pp:fmq}
Suppose that $\alpha_1, \alpha_2 \in \Phi_1(\mathscr{C})$. Then 
\begin{equation*}
%\label{eq:fmg1}
\langle \alpha_1, \phi(\alpha_2) \rangle \langle \alpha_2, \phi(\alpha_1)\rangle
\geq (\phi_1(\alpha_1),\phi_1(\alpha_2))^2.
\end{equation*}
\end{proposition}
\begin{proof}
Since both $\langle\, ,\, \rangle $ and $(\,,\,)$ are $W$-invariant, and $\phi$,
$\phi_1$ are $W$-equivariant, we may replace $\alpha_1$ and $\alpha_2$ by
$u\alpha_1$ and $u \alpha_2$ for a suitable $u\in W$ so that $\alpha_2
=\alpha_s$ for some $s\in S$. Furthermore, replace $\alpha_1$ by $-\alpha_1$ if
needs be, we may assume that $\alpha_1 \in \Phi^+_1(\mathscr{C})$. We proceed with an induction
on the depth of $\alpha_1$.

If $\dep_{\mathscr{C},1}(\alpha_1) =1$, then $\alpha_1 =\lambda \alpha_r$, 
for some positive constant $\lambda$ and some $r\in S$. 
It follows from Lemma~\ref{lem:coeff0}~(ii) that
$\phi_1(\alpha_1) = \gamma_r$. Furthermore, Lemma \ref{lem:inv} yields that
$\phi(\alpha_1) = \tfrac{1}{\lambda} \beta_r$, and hence 
\begin{multline*}
\langle \alpha_1, \phi(\alpha_2) \rangle \langle \alpha_2, \phi(\alpha_1)
\rangle = \lambda \langle \alpha_r, \beta_s \rangle \tfrac{1}{\lambda} \langle
\alpha_s, \beta_r \rangle\\
= \langle \alpha_r, \beta_s \rangle \langle \alpha_s,
\beta_r \rangle 
=(\gamma_r, \gamma_s)^2 
=(\phi_1(\alpha_1), \phi_1(\alpha_2)).
%&=(\gamma_r, \gamma_s)^2 \ \  \text{    (by the definition of $(\,,\,)$) }\\
%&=(\phi_1(\alpha_1), \phi_1(\alpha_2)).
\end{multline*}
Thus we may assume that $\dep_{\mathscr{C},1}(\alpha_1) > 1$. First, if $\langle \alpha_1,
\beta_s \rangle >0$ then Lemma \ref{lem:1.33} yields that $r_s \alpha_1 \prec_1
\alpha_1$, and hence 
\begin{align*}
\langle \alpha_1, \beta_s \rangle \langle \alpha_s,
\phi(\alpha_1)\rangle
&= \langle r_s \alpha_1, r_s \beta_s \rangle \langle r_s \alpha_s, r_s\phi(
\alpha_1) \rangle\\
&=(- \langle r_s\alpha_1, \beta_s \rangle)(- \langle \alpha_s,
\phi(r_s \alpha_1) \rangle) \\
 &\geq (\phi_1(r_s \alpha_1), \gamma_s)^2 \ \ \ \text{  (by the inductive
hypothesis)}\\
&= (\phi_1(\alpha_1), -\gamma_s)^2\\
&= (\phi_1(\alpha_1), \gamma_s)^2,
\end{align*}
as required. Thus we may further assume that $\langle \alpha_1, \beta_s \rangle
\leq 0$.

Next let $t\in S$ be such that $r_t \alpha_1 \prec_1 \alpha_1$. Then Lemma
\ref{lem:1.33} yields that  $\langle \alpha_1, \beta_t \rangle >0$, and, in
particular, $s\neq t$. Now let $w\in W$ be a maximal length alternating product of
$r_s$ and $r_t$ ending in $r_t$ such that $\dep_{\mathscr{C},1}(w\alpha_1) =
\dep_{\mathscr{C},1}(\alpha_1)-\ell(w)$. In particular, $w\neq 1$. 
Thus $\dep_{\mathscr{C},1}(w\alpha_1) \lneq \dep_{\mathscr{C},1}(\alpha_1)$, 
and so the inductive hypothesis yields that
\begin{align}
\label{eq:fmq2}
&\langle w \alpha_1, \beta_s \rangle \langle \alpha_s, \phi(w\alpha_1) \rangle
\geq (\gamma_s, \phi_1(w\alpha_1))^2 \\
\noalign{\hbox{and}}
\label{eq:fmq3}
&\langle w \alpha_1, \beta_t \rangle \langle \alpha_t, \phi(w\alpha_1) \rangle
\geq (\gamma_t, \phi_1(w\alpha_1))^2.
\end{align}
Next, for any reduced expression  $w=r_{s_1}\cdots r_{s_l}$ ($s_i\in S$),  
it is readily checked that $r_{s_l} \alpha_1 \prec_1
\alpha_1$, and so it follows from Lemma~\ref{lem:1.33} that
 $\langle \alpha_1, \beta_{s_l} \rangle >0$. Therefore $s_l\neq
s$, and $w$ has no reduced expression ending in $r_s$. That is, $w$ is a
product of $r_s$ and $r_t$ with strictly fewer than $m_{st}$ factors. Thus by
Lemma~\ref{lem:coeffpos} or a direct calculation as that in Lemma~\ref{lem:storder}, there are
non-negative constants $p, q$ and positive constant $\lambda$ such that 
$$w\gamma_s = p \gamma_s + q \gamma_t \  \text{ and } \  w\alpha_s = p\alpha_s
+\lambda q \alpha_t \  \text{ and } \  w\beta_s = p\beta_s + \tfrac{q}{\lambda}
\beta_t.$$
Thus 
\begin{align*}
&\qquad\langle \alpha_1, \beta_s \rangle \langle \alpha_s, \phi(\alpha_1)
\rangle -(\gamma_s, \phi_1(\alpha_1))^2\\
&=\langle w \alpha_1, w\beta_s \rangle \langle w\alpha_s, \phi(w\alpha_1)\rangle
-(w\gamma_s, \phi_1(w\alpha_1))^2\\
&\qquad\qquad\qquad\qquad\text{( since $\langle\,,\,\rangle$ and $(\,,\,)$ are
$W$-invariant)}\\
& =\langle w \alpha_1, p\beta_s +\tfrac{q}{\lambda}\beta_t\rangle \langle
p\alpha_s + \lambda q \alpha_t, \phi(w\alpha_1) \rangle -(\phi_1(w\alpha_1),
p\gamma_s+q\gamma_t)^2\\
& =\underbrace{p^2(\langle w\alpha_1, \beta_s \rangle \langle \alpha_s,
\phi(w\alpha_1)\rangle -(\phi_1(w\alpha_1), \gamma_s)^2)}_\text{A}\\
&\qquad +\underbrace{q^2(\langle w\alpha_1, \beta_t \rangle \langle \alpha_t,
\phi(w\alpha_1)\rangle -(\phi_1(w\alpha_1), \gamma_t)^2)}_\text{B}\ +\ \ C.
\end{align*}
where 
\begin{align*}
C& = pq(\ \tfrac{1}{\lambda}\langle w\alpha_1, \beta_t \rangle \langle \alpha_s,
\phi(w\alpha_1)\rangle +\lambda \langle w\alpha_1, \beta_s\rangle \langle
\alpha_t, \phi(w\alpha_1)\rangle\\
 &\qquad -2(\phi_1(w\alpha_1), \gamma_s)(\phi_1(w\alpha_1), \gamma_t) \ ).
\end{align*}
It follows from (\ref{eq:fmq2}) and (\ref{eq:fmq3}) that A and B are both
nonnegative.  It follows from the geometric mean and arithmetic mean inequality
that 
\begin{align*}
\tfrac{1}{\lambda}\langle w\alpha_1, \beta_t \rangle& \langle \alpha_s,
\phi(w\alpha_1)\rangle + \lambda \langle w\alpha_1, \beta_s \rangle \langle
\alpha_t, \phi(w\alpha_1)\rangle \\
& \geq 2 \sqrt{\langle w\alpha_1, \beta_s \rangle \langle \alpha_s,
\phi(w\alpha_1)\rangle \langle w\alpha_1, \beta_t \rangle \langle \alpha_t,
\phi(w\alpha_1)\rangle}\\
& \geq 2(\phi_1(w\alpha_1), \gamma_s)(\phi_1(w\alpha_1), \gamma_t) \ \ \ \ \text{ (by
(\ref{eq:fmq2}) and (\ref{eq:fmq3}))},
\end{align*}
that is, $C \geq 0$ as well. Therefore $\langle \alpha_1, \beta_s \rangle
\langle \alpha_s, \phi(\alpha_1)\rangle \geq (\gamma_s, \phi_1(\alpha_1))^2$, and
the desired result follows by induction.
\end{proof}

\section{Tits cones and a non-positivity result}
\label{sec:nonpos}

Let $\mathscr{C}=(\, S, V_1, V_2, \Pi_1, \Pi_2, \langle\,,\,\rangle \,)$ 
be a fixed Coxeter datum, and let $(W, R)$ be the corresponding 
Coxeter system. In this section we study a class of cones associated
to $\mathscr{C}$ that are analogous to the \emph{Tits cones} in the 
classical setting (as defined in \cite{VB71} or \cite[\S 5.13]{HM}). Furthermore, 
we investigate certain $W$-invariant sets in $V_1$ and $V_2$ that are
closely related to these cones.
The key result of this section is 
a generalization to  \cite[Proposition 1.2]{MX82} and 
\cite[Proposition 3.4]{HT97}. For this section we impose one additional 
condition on $\mathscr{C}$, namely:
\begin{itemize}
 \item [(C6)] $V_1 =\spa(\Pi_1)$ and $V_2=\spa(\Pi_2)$,
\end{itemize}
and we retain all other conventions and notation of earlier sections. 

%We call a convex subset of a real vector space a \emph{cone} 
%if it is  closed under addition and multiplication by positive scalars. 
%Suppose that $V$ is a real vector space and suppose that $C$ is a cone 
%in $V$. We set  
%$$
%C^*=\{f\in \hm(V, \R)\mid \text{$f(v)\geq 0$ for all $v\in C$} \}.
%$$
%It is readily checked that $C^*$ is a cone in $\hm(V, \R)$.
%Further, if $F$ is a cone in $\hm(V, \R)$, then we set 
%$$
%F^*=\{v\in V \mid \text{$g(v)\geq 0$ for all $g\in F$} \}.
%$$
%It is readily checked that $F^*$ is a cone in $V$.
%The following is a well-known result whose proof can be found in \cite[Notes (c), Lecture 1]{RB96}:
%\begin{lemma}
%\label{lem:doubledual}
%Let $C$ be a cone contained in a finite dimensional vector space $V$. Then 
%$C^{**} =\overline{C}$, where $\overline{C}$ denotes the closure of $C$ in $V$.
%\qed
%\end{lemma}

\begin{notation}
For each $i\in \{1, 2\}$ and $I\subseteq S$, recall the notation of
$\Pi_{iI}$ and $V_{iI}$ introduced in Remark \ref{rm:para}, and set 
$P_{iI}=\PLC(\Pi_{iI})\cup\{0\}$. When $I=S$ we write $P_i$ in place 
of $P_{iS}$. 
\end{notation}

%For each $i\in \{1, 2\}$ we may specify an $W$-action on $\hm(V_i, \R)$
%as follows: if $w\in W$ and $f_i\in \hm(V_i, \R)$,
%then $wf_i\in \hm(V_i, \R)$ is given by $(w f_i)x=f_i(w^{-1}x)$ for all $x\in V_i$.
%Furthermore, for $I\subset S$ recall the notation $W_I$, introduced in Remark ~\ref{rm:para}, 
%of the standard parabolic subgroup of $W$ corresponding to $I$, it is clear that $W_I$
%acts (faithfully) on $V_{iI}$. This allows us to specify an $W_I$-action on $\hm(V_{iI}, \R)$
%as follows: if $w\in W$ and $g\in \hm(V_{iI}, \R)$ then $wg \in \hm(V_{iI}, \R)$ is given 
%by $(wg)v=g(w^{-1}v)$ for all $v\in V_{iI}$.

For each $i\in \{1, 2\}$ and subset $I$ of $S$ recall the notation $W_I$ 
(introduced in Remark ~\ref{rm:para}) 
of the standard parabolic subgroup of $W$ corresponding to $I$, it is clear that $W_I$
acts (faithfully) on $V_{iI}$. This allows us to specify a $W_I$-action on $\hm(V_{iI}, \R)$
as follows: if $w\in W_I$ and $g\in \hm(V_{iI}, \R)$ then $wg \in \hm(V_{iI}, \R)$ is given 
by $(wg)v=g(w^{-1}v)$ for all $v\in V_{iI}$. Naturally, when $I=S$ the Coxeter group $W$ acts on 
$\hm(V_i, \R)$ in a similar way.

\begin{notation}
\label{def:tits}
For each $i\in \{1, 2\}$ and subset $I$ of $S$ we set
\begin{align*}
 P^\#_{iI}&=\{\,f\in \hm(V_{iI}, \R)\mid f(x)\geq 0 \text{ for all $x\in P_{iI}$}\,\},\\
\noalign{\hbox{and}}
 P_{iI}^{\#\#}&=\{\, x\in V_{iI}\mid f(x)\geq 0\text{ for all $f\in P_{iI}^{\#}$}\}.
\end{align*}
Moreover, we set $U_{iI}=\bigcup_{w\in W_I} w P_{iI}^\#$, and write
$$
U_{iI}^\#=\{x\in V_{iI}\mid f(x)\geq 0 \text{ for all $f\in U_{iI}$}\}.
$$
When $I=S$ we write $P_i^\#$, $P_i^{\#\#}$, $U_i$ and $U_i^\#$ in places of $P_{iS}^\#$, 
$P_{iS}^{\#\#}$, $U_{iS}$ and $U_{iS}^\#$ respectively.
\end{notation}

We call a (convex) subset of a real vector space a \emph{cone} 
if it is  closed under addition and multiplication by positive scalars. 
It is clear that $P_{iI}$, $P_{iI}^\#$ and $P_{iI}^{\#\#}$ ($i=1,2$) are cones for each $I\subseteq S$.   

\begin{lemma}
For each $i\in \{1, 2\}$ and each $f\in \hm(V_i, \R)$, set 
$$
\Neg(f)=\{\widehat{x}\in \widehat{\Phi_i}\mid \text{ $x\in \Phi_i^+$ and $f(x)<0$}\}.
$$
Then $U_i=\{f\in \hm(V_i, \R)\mid |\Neg(f)|<\infty  \}$.
\end{lemma}
\begin{proof}
 It is enough to prove the $U_1$ case.
Let $f\in U_1$ be arbitrary. Then $f=wg$ for some $w\in W$ and $g\in P_1^\#$. 
Suppose that $x\in \Phi_1^+$ such that $f(x)<0$. Then $(wg) x=g(w^{-1}x)<0$. 
Since $g\in P_1^\#$, it follows that $w^{-1}x\in \Phi_1^-$, that is, 
$\widehat{x}\in N_1(w^{-1})$. Now since $N_1(w^{-1})$ is of finite size, it follows 
that $U_1\subseteq \{f\in \hm(V_1, \R)\mid |\Neg(f)|<\infty  \}$. Conversely, suppose 
that $f\in \hm(V_1, \R)$ with $|\Neg(f)|<\infty$. If $|\Neg(f)|=0$ then 
$f\in P_1^\#\subseteq U_1$. Thus we may assume that $|\Neg(f)|>0$, and proceed with 
an induction. Observe that if  $|\Neg(f)|>0$ then there exists some $s\in S$ such that 
$f(\alpha_s)<0$. It is then readily checked that $|\Neg(r_s f)|=|\Neg(f)|-1$, and hence
it follows from our inductive hypothesis that $r_s f\in U_1$. Since $U_1$ is clearly 
$W$-invariant, it follows that $f\in U_1$, and therefore 
$\{f\in \hm(V_1, \R)\mid |\Neg(f)|<\infty  \} \subseteq U_1$, as required.
\end{proof}

The above lemma yields that $U_1$ and $U_2$ are cones. In fact, 
for each $I\subseteq S$ we can show that $U_{1I}$ and $U_{2I}$ are cones. 
These cones generalize the notion of the Tits cones as defined, for example, in \cite[\S 5.13]{HM}.  

\begin{definition}
 We call $U_1$ and $U_2$ the \emph{Tits cones} of the Coxeter datum $\mathscr{C}$; and
for each $I\subseteq S$ we call $U_{1I}$ and $U_{2I}$ the Tits cones of the the Coxeter
datum $\mathscr{C}_I=(\, I, V_{1I}, V_{2I}, \Pi_{1I}, \Pi_{2I}, \langle, \rangle_I\,)$
(where $\langle, \rangle_I$ is the restriction of $\langle, \rangle$ to $V_{1I}\times V_{2I}$).
\end{definition}

For each $i\in \{1, 2\}$ observe that:
\begin{equation}
\label{eq:cone}
\begin{split}
U_i^\# &= \{\,v\in V_i \mid \text{$(wf) v \geq 0$ for all $f\in P_i^\#$ and $w\in
W$ }\,\}\\
%&=\{\,v\in V_i\mid \text{$f(w^{-1}v)\geq 0$ for all $f\in P_i^*$ and $w\in
%W$}\,\}\\
&=\bigcap_{w\in W}\{\,v\in V_i \mid \text{$f(w^{-1}v)\geq 0$ for all
$f\in P_i^\#$}\,\}\\
&=\bigcap_{w\in W}\{\,wv\in V_i \mid \text{$f(v)\geq 0$ for all $f\in
P_i^\#$}\,\}\\
%&=\bigcap_{w\in W}\{\,wv\in V_i \mid \text{$v\in (P_i(V_i)^*)(V_i)^*$}\,\}.
&=\bigcap_{w\in W} wP_i^{\#\#},
\end{split}
\end{equation}
and similarly, 
\begin{equation}
 \label{eq:UI_i}
U_{iI}^\#=\bigcap_{w\in W_{I}} w P_{iI}^{\#\#}, \text{ for each 
$I\subseteq S$}. 
\end{equation}

The following is a well-known result (a proof can be found in 
\cite[Notes (b) and (c), Page 4]{RB96}).

\begin{lemma}
\label{lem:ex&cl}
Let $I\subseteq S$ be finite. For each $i\in \{1, 2\}$,  
the condition  $0\notin \PLC(\Pi_{iI})$ is equivalent to the existence of
an $f_i\in \hm(V_{iI}, \R)$ such that $f_i(x) >0$ for all $x\in \Pi_{iI}$. 
Furthermore, $P_{iI}^{\#\#}=P_{iI}$.
\qed
\end{lemma}

%Further, if $I$ is  finite then   
%it follows form Lemma \ref{lem:doubledual} that
%$P_{iI}^{**}=\overline{P_{iI}}$, (the closure of $P_{iI}$ in $V_{iI}$).
%It is easily seen that $P_{iI}$ is
%topologically closed in $V_{iI}$, and hence  
%\begin{equation}
%\label{eq:UI_i}
%U_{iI}^*=\bigcap\limits_{w\in W_{I}} wP_{iI},
%\end{equation}
%whenever $I$ is a finite subset of $S$.

\begin{lemma}
\label{lem:P**}
For each subset $I$ of $S$ and for each $i\in \{1, 2\}$ we have
$$P_i^{\#\#}\cap V_{iI} \subseteq P_{iI}^{\#\#}.$$ 
\end{lemma}
\begin{proof}
For each $i\in \{1, 2\}$ write $V_i =V_{iI}\oplus
V_{iI}'$ where $V_{iI}'$ denotes 
a (vector space) complement of $V_{iI}$ in $V_i$. Clearly, every $g_i \in
P_{iI}^\#$ gives rise to a linear functional $g_i'$ in $P_i^\#$ as follows: for
any $v \in V_i$ there is a unique decomposition of the form $v=v_{iI} +v_{iI}'$ where
$v_{iI}\in V_{iI}$ and $v_{iI}'\in V_{iI}'$, and we simply set 
$g_i'(v)=g_i(v_{iI})$. Now let $v\in P_i^{\#\#}\cap V_{iI}$, and 
 $f\in P_{iI}^\#$ be arbitrary. Then $f(v)=f'(v)\geq 0$ (because $f'\in P_i^\#$ and $v\in
P_i^{\#\#}$), and hence it follows that
$v\in P_{iI}^{\#\#}$, as required.
\end{proof}

Let $x_i\in V_i$ ($i=1, 2$) be arbitrary. Given condition (C6) of $\mathscr{C}$, 
it follows that $x_i\in V_{iI}$ for some finite subset $I$ of $S$. From this
fact we may prove the following:
\begin{lemma}
\label{lem:P_i}
Suppose that $x_i \in U_i^\#$ (for each $i\in \{1,2\}$). Let $I$ be a  subset of $S$
such that $x_i\in V_{iI}$. Then $x_i\in U_{iI}^\#$. Furthermore, if $I$ is finite
then $x_i\in P_{iI}$. 
\end{lemma}
\begin{proof}
Observe that for each $i\in \{1, 2\}$, 
\begin{align*}
x_i &\in U_i^\#\cap V_{iI}\\
%&=(\bigcap\limits_{w\in W} w P_i^{**})\cap V_{iI}\\
    &\subseteq\bigcap\limits_{w\in W_{I}}(w P_i^{\#\#}\cap V_{iI})\\
    &=\bigcap\limits_{w\in W_{I}}w\ ( P_i^{\#\#}\cap V_{iI}) \ &\text{(since
$W_{I}$ preserves $V_{iI}$)}\\
    &\subseteq\bigcap\limits_{w\in W_{I}} w P_{iI}^{\#\#} =U_{iI}^\#&\text{(by Lemma
\ref{lem:P**})}.
\end{align*}
Finally, it follows from Lemma \ref{lem:ex&cl} that 
$x_i\in \bigcap_{w\in W_{I}}w P_{iI}\subseteq P_{iI}$ whenever $I$ is finite.
%Clearly $P_{iI}$ is closed, thus when $I$ is finite, it follows from Lemma
%\ref{lem:doubledual} that $P_{iI}^{**} =P_{iI}$, and hence 
%$
%x_i\in \bigcap_{w\in W_{I}} wP_{iI}\subseteq P_{iI},
%$
%as required.
\end{proof}

Now we are ready for the central result of this section:
\begin{theorem}
\label{th:Tits}
$\langle v_1, v_2\rangle \leq 0$, for all $v_1\in U_1^\#$ and $v_2\in U_2^\#$.
\end{theorem}
\begin{proof}
Suppose for a contradiction that there are $v_1\in U_1^\#$ and $v_2\in U_2^\#$ 
with $\langle v_1, v_2\rangle >0$. Replace $v_2$ by a positive scalar
multiple of itself if necessary, we may assume that $\langle v_1, v_2 \rangle
=1$. Let $I$ be a finite subset of $S$ with $v_1\in V_{1I}$, and let
$J$ be a finite subset of $S$ with $v_2\in V_{2J}$. Next, set 
$K=I\cup J$.  
Then Lemma~\ref{lem:P_i} yields that $v_1\in U_{1K}^\#$, and $v_2\in U_{2K}^\#$.
Since $K$ is finite, it follows from Lemma~\ref{lem:ex&cl} 
that there are linear functionals $f_1\in \hm(V_{1K}, \R)$ and $f_2\in \hm(V_{2K}, \R)$ 
such that $f_1(\alpha)>1$ for all $\alpha\in \Pi_{1K}$, and $f_2(\beta)>1$ for all $\beta\in \Pi_{2K}$. 
Now set 
\begin{align*}
\mathscr{A} =\{\,& x\in U_{2K}^\#\mid \text{$f_2(x)\leq f_2(v_2)$ and 
$\langle z, x\rangle \geq 1$}\\ &\qquad\qquad\text{for some $z\in U_{1K}^\#$
with $f_1(z)\leq f_1(v_1)$}\,\}.
\end{align*}
Observe that $\mathscr{A} \neq \emptyset$, since  $v_2\in \mathscr{A}$.

Next, put $\epsilon =\frac{2}{|K|f_1(v_1)}$. We claim that for any given $x\in
\mathscr{A}$, there exists $y\in \mathscr{A}$ with $f_2(y)\leq f_2(x)-\epsilon$.
If we could prove this claim, then starting with $x=v_2$, a finite repetition 
of this process will produce some 
$y\in \mathscr{A}\subseteq U_{2K}^\#\subseteq P_{2K}$ with $f_2(y)$ 
being negative, contradicting the fact that 
$f_2(y)\in  f_2(P_{2K})\subseteq (0, \infty)$. Thus all that remains to do is 
to prove the above claim. Given arbitrary $x\in \mathscr{A}$, let 
$z=\sum_{\alpha\in \Pi_{1K}}\lambda_\alpha \alpha\in U_{1K}^\#$ be such
that $\langle z, x\rangle \geq 1$, $f_1(z)\leq f_1(v_1)$, and 
$\lambda_\alpha \geq 0$ for all $\alpha\in \Pi_{1K}$. Note that these
conditions imply that 
$\sum_{\alpha\in \Pi_{1K}}\lambda_\alpha \langle \alpha, x\rangle \geq 1$, 
which in turn implies that 
$\lambda_{\alpha_{s_0}} \langle \alpha_{s_0}, x\rangle \geq \frac{1}{|K|}$ 
for some $s_0\in K$. Then
$$
\langle \alpha_{s_0}, x\rangle \geq \frac{1}{\lambda_{\alpha_{s_0}}|K|}\geq
\frac{1}{f_1(v_1)|K|}=\frac{\epsilon}{2},
$$
since $\lambda_{\alpha_{s_0}}\leq f_1(z)\leq f_1(v_1)$. 
Set $y=r_{s_0}x$. Observe that (\ref{eq:UI_i})
indicates that $U_{2K}^\#$ is $W_{K}$-invariant, and given that $r_{s_0}\in
W_{K}$, we have $y\in U_{2K}^\#$. Moreover, 
$$
f_2(y)=f_2(x)-2\langle \alpha_{s_0}, x\rangle f_2(\beta_{s_0})<f_2(x)-\epsilon<f_2(v_2).
$$
Thus to establish our claim, we only need to show that $y\in \mathscr{A}$, and this in turn amounts to
finding some $t\in U_{1K}^\#$
such that $\langle t, y\rangle \geq 1$, and $f_1(t)\leq f_1(v_1)$. First,
suppose that $\langle z, \beta_{s_0}\rangle \geq 0$. Put $t=r_{s_0}z$.
Since $U_{1K}^\#$ is $W_{K}$-invariant and $z\in U_{1K}^\#$, it follows that
$t\in U_{1K}^\#$. Moreover, 
$$
\langle t, y\rangle =\langle r_{s_0}z, r_{s_0}x\rangle =\langle
z, x\rangle\geq 1
$$
and 
$$
f_1(t)=f_1(z)-2\langle z, \beta_{s_0}\rangle f_1(\alpha_{s_0})\leq f_1(z)\leq f_1(v_1),
$$
thus proving $y\in \mathscr{A}$ in the case $\langle z, \beta_{s_0}\rangle \geq
0$. Next, suppose that $\langle z, \beta_{s_0}\rangle <0$. Then $t=z$ will do,
indeed, 
\begin{align*}
\langle z, y\rangle &=\langle z, x-2\langle \alpha_{s_0}, x\rangle \beta_{s_0}\rangle\\
                    &=\langle z, x\rangle -2\langle \alpha_{s_0}, x\rangle \langle z,
\beta_{s_0}\rangle\\
                    &\geq \langle z, x\rangle \\
                    &\geq 1, 
\end{align*}
and by our construction, $z\in U_{1K}^\#$ and $f_1(z)\leq f_1(v_1)$, thus proving
$y\in \mathscr{A}$ in the case $\langle z, \beta_{s_0}\rangle <0$ too. This
completes the proof of the claim, and hence the theorem follows.
\end{proof}

\section{Acknowledgments}
A few results presented in this paper are taken from the author's PhD
thesis~\cite{FU0} and the author wishes to thank Prof.~R.~B.~Howlett for all
his help and encouragement throughout the author's PhD candidature. The author
also wishes to thank Prof.~G.~I.~Lehrer and Prof. R.~Zhang  for supporting this work.
Due gratitude must also be paid to Prof. W.~A.~Casselman for his helpful comments
and suggestions.

\bibliographystyle{amsplain}

\providecommand{\bysame}{\leavevmode\hbox to3em{\hrulefill}\thinspace}
\providecommand{\MR}{\relax\ifhmode\unskip\space\fi MR }
\providecommand{\MRhref}[2]{%
  \href{http://www.ams.org/mathscinet-getitem?mr=#1}{#2}
}
\providecommand{\href}[2]{#2}

\end{document}